\documentclass[12pt]{article}%

\usepackage{amssymb,amsfonts,amsmath,amsthm}
\usepackage{comment}
\usepackage{algorithm} 
\usepackage{algpseudocode}
\usepackage[nohead]{geometry}
\usepackage{setspace}
\RequirePackage[colorlinks,citecolor=blue,urlcolor=blue,linkcolor=blue]{hyperref}
\usepackage[authoryear]{natbib}

\theoremstyle{plain}
\newtheorem{thm}{Theorem}[section]

\newtheorem{lem}{Lemma}[section]
\newtheorem{cor}{Corollary}[section]

\newtheorem{assmp}{Assumption}[section]

\theoremstyle{remark}
\newtheorem{exm}{Example}[section]
\newtheorem{rmk}{Remark}[section]

\makeatletter

\begin{document}

\newcommand{\argmin}{\mathrm{argmin}}
\newcommand{\argmax}{\mathrm{argmax}}
\newcommand{\diag}{\mathrm{diag}}
\newcommand{\ddiag}{\mathrm{ddiag}}
\newcommand{\rank}{\mathrm{rank}}
\newcommand{\Id}{\mathrm{Id}}
\newcommand{\Arg}{\mathrm{Arg}}
\newcommand{\bb}{\mathbb}
\newcommand{\SVD}{\mathrm{SVD}}
\newcommand{\etheta}{e^{i \theta}}
\newcommand{\mtheta}{\min_{\theta}}
\newcommand{\poly}{\mathrm{poly}}
\newcommand{\polylog}{\text{polylog } n}
\newcommand{\diam}{\mathrm{diam}}
\newcommand{\rad}{\mathrm{rad}}

\renewcommand{\Re}{\operatorname{Re}}
\renewcommand{\Im}{\operatorname{Im}}
\renewcommand{\hat}{\widehat}
\renewcommand{\cal}{\mathcal}
\renewcommand{\tilde}{\widetilde}

\makeatletter
\def\@biblabel#1{\hspace*{-\labelsep}}
\makeatother

\title{Eigenvector Under Random Perturbation: A Nonasymptotic Rayleigh-Schr\"{o}dinger Theory}
\author{Yiqiao Zhong \thanks{E-mail: \textit{yiqiaoz@princeton.edu}}\vspace{0.1in}\\
  Princeton University}

\date{}
\maketitle
\sloppy

\begin{abstract}
Rayleigh-Schr\"{o}dinger perturbation theory is a well-known theory in quantum mechanics and it offers useful characterization of eigenvectors of a perturbed matrix. Suppose $A$ and perturbation $E$ are both Hermitian matrices, $A^t = A + tE$, $\{\lambda_j\}_{j=1}^n$ are eigenvalues of $A$ in descending order, and $u_1, u^t_1$ are leading eigenvectors of $A$ and $A^t$. Rayleigh-Schr\"{o}dinger theory shows asymptotically, $\langle u^t_1, u_j \rangle \propto t / (\lambda_1 - \lambda_j)$ where $ t = o(1)$. However, the asymptotic theory does not apply to larger $t$; in particular, it fails when $ t \| E \|_2 > \lambda_1 - \lambda_2$. In this paper, we present a nonasymptotic theory with $E$ being a random matrix. We prove that, when $t = 1$ and $E$ has independent and centered subgaussian entries above its diagonal, with high probability, 
\begin{equation*}
| \langle u^1_1, u_j \rangle | = O(\sqrt{\log n} / (\lambda_1 - \lambda_j)),
\end{equation*}
for all $j>1$ simultaneously, under a condition on eigenvalues of $A$ that involves all gaps $\lambda_1 - \lambda_j$. This bound is valid, even in cases where $\| E \|_2 \gg \lambda_1 - \lambda_2$. The result is optimal, except for a log term. It also leads to an improvement of Davis-Kahan theorem.
\end{abstract}


\section{Introduction}\label{sec::intro}
Suppose $A \in \bb{C}^{n \times n}$ is a Hermitian matrix, and its spectral decomposition is given by 
\begin{equation*}
A = U\Lambda U^* = \sum_{j=1}^n \lambda_j u_j u_j^*,
\end{equation*}
where $\Lambda = \diag(\lambda_1,\ldots, \lambda_n)$, $\lambda_1 \ge \lambda_2 \ge \ldots \ge \lambda_n$, and $U := (u_1,\ldots,u_n)$ is a unitary matrix. Suppose $E \in \bb{C}^{n \times n}$ is a random Hermitian matrix, and $\tilde{A} = A + E$ is the perturbed matrix. Let $\tilde{u}_1$ be the leading eigenvector of $\tilde{A}$, which is of primary interest in this paper. To be specific, we want to derive sharp bounds on $\langle \tilde{u}_1, u_j \rangle$.

The study of eigenvector perturbation has proved to be quite useful in many fields, including statistics, applied mathematics, theoretical computer science, quantum mechanics and so on. More often than not, $A$ is an ideal matrix, usually assumed to have a relatively simple structure, and $E$ represents measurement errors, sampling errors or disturbance, generally interpreted as noise. Given the importance of spectral decomposition, it is, thus, desirable to understand how eigenvectors are perturbed. 

For example, in community detection \citep{RohChaYu11}, $A$ is usually a low rank matrix that encodes the information of unknown labels of vertices (which community a vertex belongs to), and spectral algorithms try to recover the label information from an observed random connectivity graph by computing top eigenvectors of $\tilde{A}$ (adjacency matrix, graph Laplacian, etc.). In angular synchronization \citep{Sin11}, in order to estimate unknown angles $\theta_1,\ldots, \theta_n$ from noisy measurements $y_{jk}$ of offsets $\theta_j - \theta_k$, one practical approach is to build symmetric matrix $H = (H_{jk})$ where $H_{jk} = \exp(iy_{jk})$ and then extract the leading eigenvector from $H$.    

A powerful tool for bounding the error $\tilde{u}_1 - u_1$ is Davis-Kahan theorem \citep{DavKah70}. In the simplest form, Davis-Kahan theorem bounds the $\ell^2$ norm of the error as follows:
\begin{equation}\label{ineqn::DK}
\| \tilde{u}_1 - u_1 \|_2 = O\Big( \frac{ \| E \|_2 }{\delta} \Big),
\end{equation}
where $\| E \|_2 = \max_{\| u \|_2 \le 1} \| E u \|_2$ is the matrix operator norm, and $\delta := \lambda_1 - \lambda_2$ is the gap between the first two eigenvalues of $A$. Note the global phase of $\tilde{u}_1$ has to be suitably chosen, since the leading eigenvector is not unique. In general cases where multiplicity of $\lambda_1$ is possibly greater than 1, Davis-Kahan theorem provides a bound for the error measured in terms of eigenspaces. 
The Davis-Kahan theorem enjoys great popularity in fields such as statistics and computer science, where it is commonly employed to study principal component analysis, spectral algorithm and so on \citep{Von07,RohChaYu11, fan2013large}.

Its great applicability has led some people to examine the theorem more carefully, and particularly two interesting questions arise: (i) Is there any bound tighter than (\ref{ineqn::DK})? (ii) Can we characterize the error $\tilde{u}_1 - u_1$ in a more delicate fashion, beyond an $\ell^2$ bound?

For question (i), in general it is impossible to see improvement over (\ref{ineqn::DK}). Rayleigh-Schr\"{o}dinger perturbation theory, as we will soon discuss, gives first-order asymptotic expansion of $\tilde{u}_1 - u_1$ that shows generally $O(\| E \|_2 / \delta)$ is indispensable. In particular, usually one has to require $\delta \gg \| E \|_2$ to make the error $\tilde{u}_1 - u_1$ vanish. To see the necessity of this condition, consider $E = (\lambda_1 - \lambda_2 + \varepsilon) u_2 u_2^*$, where $\varepsilon > 0$; in this example, the perturbation results in a switch of order of the first two eigenvectors, so in general when $\delta$ and $\| E\|_2$ are of the same magnitude, there is no hope for $\|\tilde{u}_1 - u_1 \|_2 \to 0$. One might hope to see the condition $\delta \gg \| E \|_2$ weakened when $E$ is a random matrix. However, even when $E$ is random, it is known that (\ref{ineqn::DK}) cannot be improved in general: in fact, when $A$ has rank 1 and $E$ is Gaussian Orthogonal Ensemble (GOE) or Gaussian Unitary Ensemble (GUE), there is a phase transition: when $E$ is properly scaled such that $\| E \|_2 \to 2$ a.s., we have $|\langle \tilde{u}_1, u_1 \rangle|^2 \to \big( 1 - (1/\lambda_1)^2 \big)_+$ a.s.  \citep{Ben11}.     


Nevertheless, interesting results are obtained by \cite{Vu11} and \cite{ORVuWan13}, in which an improvement over Davis-Kahan theorem is indeed possible when the matrix $A$ is of low rank and $\lambda_1 \gg \delta$. Under favorable regime, $\| \tilde{u}_1 - u_1\|_2 = o_P(1)$ even when $\| E \|_2 \gg \delta$.   These results are strong indications that we should include other eigenvalues in the bound on $\| \tilde{u}_1 - u_1\|_2$, not merely the eigengap $\delta$, especially when the eigenvalues exhibit different scales. Understandably, when the eigenvector perturbation is inhomogeneous,
we should look at all eigenvalues rather than simply the eigengap $\delta$. 

For question (ii), Rayleigh-Schr\"{o}dinger perturbation theory \citep{Ray96, Sch26}, which is well-known in quantum mechanics, provides asymptotic characterization of the eigenvector perturbation. It also explains, in a sense, the mysteries in \cite{Vu11} and \cite{ORVuWan13}. For $t \in [0,1]$, let $A^t = A + t E$ be an interpolation of $A$ and $\tilde{A}$, and the leading eigenvector of $A^t$ is denoted by $u^{t}_1$ accordingly. When $t=1$, $A^t$ is simply $\tilde{A}$ and $u^{t}_1$ is $\tilde{u}_1$. We can choose the global phase of $u_1^t$ appropriately, such that for any $j > 1$, 
\begin{equation}\label{eqn::RS}
\langle u^t_1 - u_1, u_j \rangle = \langle u^t_1, u_j \rangle = \frac{u_j^* E u_1}{\lambda_1 - \lambda_j} t + o(t).
\end{equation}
The asymptotic result shows the amount of projection of $u^t_1 - u_1$ onto $\mathrm{span}\{u_j\}$ is inversely proportional to $\lambda_1 - \lambda_j$.  Note it leads to (\ref{ineqn::DK}) in the asymptotic sense, since we can use a trivial inequality $\lambda_1 - \lambda_j \ge \delta$ and combine dominant terms in all projections to derive an $\ell^2$ bound. But (\ref{ineqn::DK}) is tight only when these `gaps' $\lambda_1 - \lambda_j$ are comparable. When $\lambda_1 - \lambda_j$ exhibits different scales, a bound such as (\ref{ineqn::DK}) that involves only the eigengap $\delta = \lambda_1 - \lambda_2$, is too crude. Informative and precise as it is, however, (\ref{eqn::RS}) is based on asymptotics, and cannot handle scenarios where $t \| E \|_2 \gg \delta$. This asymptotic theory falls apart when, for example, $E = (\lambda_1 - \lambda_2 + \varepsilon) u_2 u_2^*$, which we have encountered before.


The main message of this paper is to show, when $E$ is a random matrix, the pathological example is rare, and the relation $|\langle \tilde{u}_1, u_j \rangle| \propto 1/(\lambda_1 - \lambda_j)$ is valid even when $\| E \|_2 \gg \delta$. Under a reasonable random regime, which includes GOE and GUE, and a condition on the eigenvalues involving all `gaps' $\lambda_1 - \lambda_j$ versus a single gap $\delta$, we prove
\begin{equation}\label{eqn::introMain}
|\langle \tilde{u}_1, u_j \rangle| = O(\sqrt{\log n} / (\lambda_1 - \lambda_j) ), \qquad \forall \, j > 1
\end{equation}
with high probability, where $O(\cdot)$ hides an absolute constant. The bound is true for all $j>1$ simultaneously. We also prove that the bound (\ref{eqn::introMain}) is tight, and that in general the condition on the eigenvalues cannot be weakened up to some poly-logarithmic term $O(\log^K n)$ (for some $K \in \bb{N}$).

Informally, we summarize our results as follows:
\begin{itemize}
\item $\tilde{u}$ is approximately the solution to a system of linear equations. 
\item The eigengap $\delta := \lambda_1 - \lambda_2$ can be as small as $\polylog$, a vanishing quantity compared with $\| E \|_2$. 
\item Roughly speaking, $|\langle \tilde{u}_1, u_j \rangle|$ is comparable to $1/(\lambda_1  - \lambda_j)$ for all $j > 1$.
\item In many scenarios, our bound improves upon Davis-Kahan theorem.
\end{itemize}

As a byproduct, we also have a result on perturbation of the leading eigenvalue that generalizes Weyl's inequality---see Theorem \ref{thm::eigvalPtb}.

Here is the arrangement of the paper. We first state our nonasymptotic Rayleigh-Schr\"{o}dinger (RS) perturbation result in Section \ref{sec::nonasympRS}. Near-optimality results concerning the condition and tightness of the bound are stated in Section \ref{sec::opt}. Further related works and notational definitions are in Section \ref{sec::related} and Section \ref{sec::notation}. Proofs of main results run from Section \ref{sec::pre} to Section \ref{sec::match}, and supporting lemmas are in the appendix.


\section{Main Results}\label{sec::main}
\subsection{Nonasymptotic RS perturbation theory}\label{sec::nonasympRS}
To set up our results, we first introduce the random regime and the assumption on eigenvalues of $A$, as a necessary step toward establishing our nonasymptotic Rayleigh-Schr\"{o}dinger (RS) theory.
\paragraph{Regime of random perturbation} ~\\
We suppose the random matrix $E$, under eigenbasis $\{u_1, u_2,\ldots, u_n\}$, has independent subgaussian variables above its diagonal. To be precise,  denote
\begin{equation*}
\tilde{E} = U^* E U,
\end{equation*}
We suppose $\{ \Re(\tilde{E}_{jk}), \Im(\tilde{E}_{jk}) \}_{j \le k}$ are jointly independent and have zero mean; and for any $j \le k$ the subgaussian norms of $\Re(\tilde{E}_{jk}), \Im(\tilde{E}_{jk})$ satisfy $\| \Re(\tilde{E}_{jk}) \|_{\psi_2} \le 1, \| \Im(\tilde{E}_{jk}) \|_{\psi_2} \le 1$. Here, $\Re(a)$ and $\Im(a)$ denote the real part and imaginary part of any $a \in \bb{C}$, and for a subgaussian random variable $\xi$, the subgaussian norm is defined as $\| \xi \|_{\psi_2} := \sup_{q \ge 1} q^{-1/2} ( \bb{E} |\xi|^q )^{1/q}$.

As a special case, the perturbation $E$ can be a GOE or a GUE, since they are invariant under any unitary matrix $U$. We remark that if we assume $A$ is a diagonal matrix, then we can choose $U$ to be an identity matrix. This means $\tilde{E} = E$, and the independence assumption is directly made on $E$, and also $\langle \tilde{u}_1, u_j \rangle$ is just the $j$'th entry of $\tilde{u}_1$. It is important to state that, in fact, we will not lose anything by assuming $A$ is diagonal, since we can first prove everything for the diagonal case and the general case will follow immediately. 

\paragraph{Condition on eigenvalues} ~\\
Let $d(\lambda) \in \bb{R}^{n-1}$ be a vector with $[d(\lambda)]_j= 1 / (\lambda_1-\lambda_{j+1})$. The vector norm $\| \cdot \|_{p/(p-2)}$ is interpreted as the infinity norm $\| \cdot \|_\infty$ when $p=2$, and as the $1$-norm $\| \cdot \|_1$ when $p = \infty$. Suppose $c_0>0$ is some small absolute constant throughout the paper. 

\begin{assmp}\label{assmp::lambda}
Suppose either of the following: (1) there exists $p \in [2,\infty)$ such that 
\begin{equation}\label{ineqn::eigvalCon}
\sqrt{p\log n}\, n^{1/p}\| d(\lambda) \|_{p/(p-2)} \le c_0, 
\end{equation}
(2)
\begin{equation}\label{ineqn::eigvalCon2}
\log n \,\| d(\lambda) \|_{1} \le c_0.
\end{equation}
\end{assmp}

Informally, we assume $\| d(\lambda) \|_{p/(p-2)} \ll n^{-1/p}$. This assumption incorporates information from all `gaps' $\lambda_1 - \lambda_j$, rather than a single eigengap $\delta = \lambda_1 - \lambda_2$. The vector $d(\lambda)$ is a reasonable quantity to expect since, according to RS perturbation theory, asymptotically $|\langle \tilde{u}_1, u_1 \rangle| \propto [d(\lambda)]_j$. When  $p = 2$, (\ref{ineqn::eigvalCon}) is equivalent to $\delta = \Omega( \sqrt{n \log n} )$, so it regresses to the usual eigengap condition for Davis-Kahan theorem ($\delta \gg \| E \|_2 = O_P(\sqrt{n})$); but for other choice of $p$, this assumption allows for more flexibility. When other `gaps' $\lambda_1 - \lambda_j$ ($j > 2$) are significantly larger, it is possible that $\delta \ll \| E \|_2$. We will soon discuss many interesting examples this assumption encompasses.

We remark that we can only look for $p$ in the range $[2, \log n]$, with the only slight difference being a modification of constant $c_0$. Indeed, when $p = \log n, \| d(\lambda) \|_{p/(p-2)} \asymp \| d(\lambda) \|_1$ and $n^{1/p} \asymp 1$, so there is no need to seek very large $p$.



%
%

Now we are in a position to state the main result.
\begin{thm}\label{thm::main}
Suppose $n \ge 2$ and Assumption \ref{assmp::lambda}. Then, with probability $1 - O(n^{-1})$, 
\begin{equation}\label{ineqn::main1}
|\langle \tilde{u}_1, u_j \rangle| \le C\sqrt{\log n} / (\lambda_1  - \lambda_j), \qquad \forall \, 2 \le j \le n,
\end{equation}
where $C>0$ is an absolute constant. 
\end{thm}

Compared with the usual RS perturbation theory, there is an additional $\sqrt{\log n}$ factor, because we need to bound $|\langle \tilde{u}_1, u_j \rangle|$ simultaneously for all $j > 1$. In Section \ref{sec::opt} we show this is indeed sharp. 
A direct consequence of this bound is the $\sin \theta$ bound, which is simply the bound on the angle between two leading eigenvectors (since $\lambda_1$ has multiplicity 1). Note that $\sin \theta$ bound is a common metric for eigenspaces (see \cite{SteSun90}). Suppose that we choose the global phase of $\tilde{u}_1$ such that $ \langle \tilde{u}_1, u_1 \rangle = \big| \langle \tilde{u}_1, u_1 \rangle \big|$. We also write $K_{n,p}(\lambda)$ to denote the left-hand side of (\ref{ineqn::eigvalCon}) or (\ref{ineqn::eigvalCon2}) (depending on $p$). 

\begin{cor}\label{cor::main}
Let $\theta(\tilde{u}_1, u_1) := \arccos \langle \tilde{u}_1, u_1 \rangle$ be the angle between $\tilde{u}_1$ and $u_1$. Suppose $n \ge 2$ and Assumption \ref{assmp::lambda}. Then, holds with probability $1 - O(n^{-1})$,
\begin{equation}\label{ineqn::main2}
\sin \theta(\tilde{u}_1, u_1) \le C\sqrt{\log n}\, \big( \sum_{j \ge 2}^n (\lambda_1  - \lambda_j)^{-2} \big)^{1/2}
\end{equation}
where $C$ is the same constant as in Theorem \ref{thm::main}. As a consequence, $\sin \theta(\tilde{u}_1, u_1) \le C \cdot K_{n,p}(\lambda)$.
\end{cor}

Since $\sin^2 \theta(\tilde{u}_1, u_1) = 1 - \langle \tilde{u}_1, u_1 \rangle^2$, this bound is equivalent to a bound on $\langle \tilde{u}_1, u_1 \rangle$. This is also equivalent to $\ell^2$ bound $\| \tilde{u}_1 - u_1 \|_2$ up to an absolute constant, because in general we have
\begin{equation*}
\sin \theta(\tilde{u}_1, u_1) \le \| \tilde{u}_1 - u_1 \|_2 \le \sqrt{2} \sin \theta(\tilde{u}_1, u_1).
\end{equation*}
The first bound in Corollary \ref{cor::main} is a direct consequence of Theorem \ref{thm::main}. The second bound is due to a simple lemma (Lemma \ref{lem::simpleNormBound}). When $K_{n,p}(\lambda) = o(1)$, clearly the $\ell^2$ error $\| \tilde{u}_1 - u_1 \|_2$ tends to zero with high probability. Therefore, with a much weaker condition on eigenvalues, our nonasymptotic RS theory assures $\| \tilde{u}_1 - u_1 \|_2 = o_P(1)$ even when $\| E \|_2 \gg \delta$.

\paragraph{Examples} 
\begin{exm}[large $\delta$]
Suppose the eigengap $\delta = \lambda_1 - \lambda_2$ satisfies $\delta = \Omega(\sqrt{n \log n})$, then Assumption \ref{assmp::lambda} holds with the choice $p = 2$. Thus, with $p=2$, Theorem \ref{thm::main} offers a more delicate bound than (\ref{ineqn::DK}). In particular, Corollary \ref{cor::main} gives a bound $O(\sqrt{n \log n} / \delta)$ for $\sin \theta(\tilde{u}_1, u_1)$, which recovers (\ref{ineqn::DK}) except for a logarithmic term. Arguably, the additional term is incurred because we consider simultaneous control of $\langle \tilde{u}_1, u_j\rangle$ (for any $j > 1$) in the first place.

For many random models, the perturbation satisfies $\|  E\|_2 = O_P(\sqrt{n})$, and known results \citep{Ben11,BenNad12} indicate we actually only need $\delta \asymp \sqrt{n}$ for good characterization of leading eigenvectors (phase transition etc.). These results, however, assumes that $A$ has low rank, and they cannot deal with cases where `gaps' $\lambda_1 - \lambda_j$ have different scaling.

\end{exm}

\begin{exm}[small $\delta$]
Suppose $\lambda_j = (n+1-j) \log^{2+\varepsilon}n$ for $1\le j \le n$, where $\varepsilon > 0$. For sufficiently large $n$, (\ref{ineqn::eigvalCon2}) is true so Assumption \ref{assmp::lambda} holds. Theorem \ref{thm::main} gives $\langle \tilde{u}_1, u_{j+1} \rangle = O(j^{-1}(\log n)^{-3/2 - \varepsilon})$ for $j\ge 1$, which shows the perturbation is different along different direction $u_j$. This inhomogeneity is also indicated by RS perturbation theory, a phenomenon not captured by (\ref{ineqn::DK}). 

Interestingly, the eigengap $\delta$ is as small as $\log^{2+\varepsilon} n$, which is much smaller than the magnitude of the perturbation ($\|E\|_2 = O_P(\sqrt{n})$ for GOE). The heuristic explanation is that, the perturbation $E$ affects all eigenvectors in a mild way, rather than changes a few eigenvectors severely. Bad examples such as $E = (\lambda_1 - \lambda_2 + \varepsilon) u_2 u_2^*$, which we encountered in Section \ref{sec::intro}, are very rare. Moreover, larger eigenvalues and their associated eigenvectors are less affected by the perturbation. This is also why if we desire a small $\delta$, we usually require a larger $\lambda_1$ ($\lambda_1$ scales with $n \log^{1+\varepsilon} n$ here). 



\end{exm}

\begin{exm}[low rank]\label{exm::lowrank}
Suppose the rank of $A$ is $r$, where $r < n$. Assume there is some $p\in [2, \infty)$ such that the leading eigenvalue satisfies $\lambda_1 = \Omega(\sqrt{p \log n}\, n^{1 - 1/p})$, and the eigengap satisfies $\delta = \Omega( \sqrt{p \log n}\, n^{1/p} r )$. Then Assumption \ref{assmp::lambda} holds. In particular, choosing $p = \log n$, we only need $\lambda_1 = \Omega(n\log n), \delta = \Omega(r \log n)$. When the rank is small, e.g., $r = O(1)$, the eigengap can be as small as $\delta = \Omega(\log n)$. 

It readily follows from Corollary \ref{cor::main} that $\sin \theta(\tilde{u}_1, u_1) = \tilde{O}(\sqrt{r}/\delta + \sqrt{n}/ \lambda_1)$, where $\tilde{O}$ hides a logarithmic term. Under a slightly weaker assumption on the random model, \cite{ORVuWan13} establishes the same bound for singular vectors. Section \ref{sec::related} offers a detailed discussion.

\end{exm}
 
\subsection{Optimality Results}\label{sec::opt}
In this subsection, we probe into two natural questions concerning the optimality of Theorem \ref{thm::main}: is the upper bound sharp? Can the assumption be relaxed? We give an affirmative answer to the first question (up to a constant). We also show that in general it is necessary to require $n^{1/p} \| d(\lambda) \|_{p/(p-2)} = O(\polylog)$, where $\polylog$ means $\log^K n$ for some constant $K$. This is an indication that Assumption \ref{assmp::lambda} cannot be weakened much except for a log term.

Without loss of generality, we constrain $A$ to be a diagonal matrix, i.e., $A = \diag \{\lambda_1,\ldots, \lambda_n\}$, where $\lambda_1 > \lambda_2 \ge \ldots \ge \lambda_n$. We also only consider real numbers for simplicity. Clearly the eigenvectors of $A$ are simply $\{e_j\}_{j=1}^n$, the standard basis in $\bb{R}^n$. We assume that $E$ has the following structure:
\begin{equation}\label{eqn::E-g}
E = \left( \begin{array}{cc}
0 & g^T \\ g & 0 \end{array}  \right),
\end{equation}
where $g$ is a standard normal vector $N(0,I_{n-1})$. Recall that $\tilde{u}_1$ is the leading eigenvector of $\tilde{A} = A + E$, so $\langle \tilde{u}_1, e_j \rangle$ is simply the $j$'th entry of $\tilde{u}_1$. The following theorem provides a lower bound on the entry-wise perturbation.
\begin{thm}\label{thm::main2}
Suppose $n \ge 2$ and Assumption \ref{assmp::lambda}. Assume $A = \diag \{\lambda_1,\ldots, \lambda_n\}$ with $\lambda_1 > \lambda_2 \ge \ldots \ge \lambda_n$, and $E$ has the form (\ref{eqn::E-g}). Then, with probability $1 - O(n^{-2})$, 
\begin{equation}
| \langle \tilde{u}_1, e_{j+1} \rangle | \ge \frac{ |g_j| }{4 (\lambda_1 - \lambda_{j+1})}, \qquad \forall  1 \le j \le n-1.
\end{equation}
As a consequence, with probability $1 - O(n^{-2})$ we have $\max_j (\lambda_1 - \lambda_{j+1}) | \langle \tilde{u}_1, e_{j+1} \rangle | \gtrsim \sqrt{\log n}$,  which matches the upper bound in Theorem \ref{thm::main} up to some constant.
\end{thm}

This theorem says that the upper bound in Theorem \ref{thm::main} is optimal. Roughly speaking, $| \langle \tilde{u}_1, u_j \rangle |$ has the `right' form $O(1/(\lambda_1 - \lambda_j))$, which agrees with the usual RS perturbation theory. It is possibly interesting to see if the additional $\sqrt{\log n}$ factor can be removed for a given $j$, but we do not study this problem in this paper.

Our next task is to construct an example such that Theorem \ref{thm::main} fails while $n^{1/p} \| d(\lambda) \|_{(p-2)/p} = \Theta(\polylog)$. In this example, we will see that the leading eigenvectors $\tilde{u}_1$ and $u_1$ are orthogonal to each other ($\tilde{u}_1 \bot u_1$) with high probability for large $n$, which implies that $\tilde{u}_1$ is an inconsistent estimator of $u_1$. This inconsistency contrasts with the conclusion in Corollary \ref{cor::main}, and therefore shows the near-optimality of the assumption (up to a log term).

The example is constructed as follows. Let $A$ be a diagonal matrix, i.e., $A = \{ \lambda_1, \lambda_2,\ldots, \lambda_n\}$ where $\lambda_1 > \ldots > \lambda_n$, and $E$ be a real symmetric matrix having zeros in its first row and first column. The bottom right block of $E$, denoted by $G$, is a GOE of size $n-1$, meaning the off-diagonal entries of $G$ are standard Gaussian variables $N(0,1)$, diagonal entries are $N(0,2)$, and entries above the diagonal are jointly independent. In matrix forms, we write
\begin{equation*}
A = \left( \begin{array}{ccc} \lambda_1 & \ldots & 0 \\ \vdots & \ddots & \vdots \\ 0 & \ldots & \lambda_n \end{array} \right) =: 
\left( \begin{array}{cc} \lambda_1 & 0 \\ 0 & A_0 \end{array} \right), \qquad 
E = \left( \begin{array}{cc}0 & 0 \\ 0 & G \end{array} \right).
\end{equation*}
Clearly, the leading eigenvector of $A$ is $e_1 = (1,0,\ldots,0)^T$. If, as we will prove, $\lambda_{\max}(A_0 + G) > \lambda_1$, then $\tilde{A}$ will have a leading eigenvector orthogonal to $e_1$. To this end, we set $\lambda_n = 3\sqrt{n}$. and determine other values of $\lambda_j$ by $\lambda_1 - \lambda_{j+1} = j^{(p-2)/p} \cdot n^{1/p} / \log^2 n$. Equivalently, we set
\begin{equation*}
\lambda_{j+1} = 3\sqrt{n} + ((n-1)^{(p-2)/p} - j^{(p-2)/p}) n^{1/p} / \log^2 n, \quad \forall \, j=0,1,\ldots,n-1.
\end{equation*}
With this setting, Assumption \ref{assmp::lambda} is slightly violated:
\begin{equation*}
n^{1/p} \| d(\lambda) \|_{p/(p-2)} = \log^2 n \cdot \Big( \sum_{j=1}^{n-1} \frac{1}{j} \Big)^{(p-2)/p} = \Theta(\log^{3 - 2/p} n),
\end{equation*}
whereas Assumption \ref{assmp::lambda} requires that $n^{1/p} \| d(\lambda) \|_{p/(p-2)}$ vanishes roughly as $O(1/\sqrt{\log n})$ ($O(1/\log n)$ for large $p$). The next result says Corollary \ref{cor::main} (and thus Theorem \ref{thm::main}) would be false if we weakened the assumption by a log term $O(\polylog)$.
\begin{thm}\label{thm::main3}
There exists some absolute constant $N > 0$ such that for all $n > N$, with probability $1 - O(n^{-1})$, $\lambda_{\max}(A_0 + G) > \lambda_1$. As a consequence, this implies $\tilde{u}_1 \bot u_1$.
\end{thm}

When $n^{1/p} \| d(\lambda) \|_{p/(p-2)}$ grows as slow as $O(\polylog)$, this results says the perturbation may leave $\tilde{u}_1$ as a useless estimator of $u_1$, as $\tilde{u}_1$ reveals no information about $u_1$. It is possibly of interest to ascertain the optimal condition on $n^{1/p} \| d(\lambda) \|_{p/(p-2)}$ for our main results, but we do not pursue this goal in this paper.

\subsection{Further Related Work}\label{sec::related}

In classical paper by \cite{DavKah70}, only deterministic matrices were considered. \cite{Wed72} established similar bounds for singular value decomposition, which is similar to spectral decomposition but applies to non-Hermitian matrices as well. Recently, \cite{FanWanZho16} studied $\ell^\infty$ perturbation bounds for structured low-rank matrices that have incoherent eigenvectors, but their result is suboptimal when perturbation matrix is random.

In the literature of random matrix theory and statistics, there are many results of eigenvectors for spiked Wigner matrices \citep{Ben11} and spiked covariance matrices \citep{Pau07, JohLu12, BenNad12}. These results assume that the matrix to be perturbed has low rank, and (implicitly) that the eigenvalues are of the same magnitude. Thus, these works are very different in nature with our results in this paper.

\cite{Vu11} and \cite{ORVuWan13} studied the top singular vectors (containing eigenvectors as a special case) of a perturbed (rectangular) matrix. To compare our result with theirs, suppose the matrices under consideration are symmetric. When $A$ is of low-rank, and $\lambda_1 \gg \delta$, and the perturbation is random, they proved upper bounds that improve over Davis-Kahan theorem (or Wedin's theorem \citep{Wed72}). Roughly speaking, they showed that when $\lambda_1 \ge \max\{n, \sqrt{n}\delta\}$, with large probability,
\begin{equation*}
\sin \theta(\tilde{u}_1, u_1) = O(\frac{\sqrt{r}}{\delta}).
\end{equation*}
See Theorem 9 in \cite{ORVuWan13}. They also proved a more general theorem that requires a weaker assumption on the random model. Their bound is the same as what we obtained (see Example \ref{exm::lowrank}), except for a log term.  Yet it is not clear why their result requires $\lambda_1 \ge \max\{n, \sqrt{n}\delta\}$, and how $\lambda_1 \gg \delta$ helps. Our nonasymptotic RS perturbation theory offers a natural explanation and generalizes their results (with the price of a stronger independence assumption). 

\subsection{Notations}\label{sec::notation}
For complex vectors $u,v \in \mathbb{C}^n$, denote the conjugate transpose of $u$ by $u^*$, and denote the inner product by $\langle u, v \rangle = u^* v$. Let $\mathbb{S}^{n-1}$ denote the unit ball $\{u \in \bb{C}^n: \| u \|_2 = 1\}$. We use $\diag(u)$ to mean a diagonal matrix with its diagonal entries being $u_1,\ldots,u_n$. For a complex matrix $M$, the real and imaginary parts of $M$ are $\Re(M)$ and $\Im(M)$. We use $\ddiag(M)$ to extract its diagonal entries into a vector. For matrix norms, we denote $\| M \|_p = \sup_{ \| u \|_p \le 1 } \| M u \|_p$, and $\| M \|_{p,q} = \sup_{ \| u \|_p \le 1 } \| M u \|_q$, where $p,q \in [1, \infty]$. For nonnegative numbers $a,b$, if $a \lesssim b$ (or $a \gtrsim b$), we mean there is an absolute constant $C>0$ such that $a < Cb$ (or $a > Cb$); we also write $a = O( b )$ for $a \lesssim b$ and $a = \Omega(b)$ for $a \gtrsim b$. We write $a \asymp b$ if $a \lesssim b$ and $b \lesssim a$.

\section{Preliminary Lemmas}\label{sec::pre}
Throughout this section, we will use $X \in \bb{C}^{n \times n}$ to denote a Hermitian matrix with zero mean and independent sub-gaussian entries in its upper triangular part, i.e., $\bb{E} X_{jk} = 0$, $\Re(X_{jk}), \Im(X_{jk})$ are independent for any $1 \le j \le k \le m$, and $\| \Re(X_{jk}) \|_{\psi_2} \le 1, \| \Im(X_{jk}) \|_{\psi_2} \le 1$. 

\subsection{Matrix Norm Bounds}

We will need a comparison theorem for subgaussian processes. The form of the following theorem is from \cite{Ver15} (see Theorem 8.2).

\begin{thm}[Fernique-Talagrand's comparison theorem]\label{thm::subgCmp}
Let $(T, d)$ be a separable metric space. Suppose $G_u$ is a Gaussian random process and $H_u$ is a separable sub-gaussian process on $T$ where $u \in T$. Assume that $\bb{E} G_u = \bb{E} H_u = 0$ for all $u \in T$. Assume also that for some $M > 0$, the following increment comparison holds:
\begin{equation}\label{ineqn::increment}
\| H_u - H_v \|_{\psi_2} \le M \| G_u - G_v \|_2, \qquad \forall u,v \in T.
\end{equation}
Then, for some absolute constant $C>0$,
\begin{equation*}
\bb{E} \sup_{u \in T} H_u \le CM \bb{E} \sup_{u \in T} G_u.
\end{equation*}
\end{thm}

From Fernique-Talagrand's comparison theorem, it is easy to prove Chevet's inequality \citep{Che78} for the subgaussian process $\langle u, Xu \rangle$. For completeness we include its proof in the appendix.

\begin{thm}[Chevet's inequality]\label{thm::Chevet}
Suppose $T \subset \bb{C}^{n}$ has a countable and dense subset. Then
\begin{equation*}
\bb{E} \sup_{u \in T} \langle u, Xu \rangle \le C w(T) \rad(T),
\end{equation*}
where $C>0$ is an absolute constant, $\rad(T)$ is the radius of $T$, and $w(T) := \bb{E} \sup_{u \in T} | \langle u, g \rangle |$ is the Gaussian width of $T$. Here $g$ is a standard gaussian vector in $\bb{C}^n$.
\end{thm}

From Chevet's inequality, we can give a sharp bound on the matrix norm $\| X \|_{p',p}$. Here $p'$ is the conjugate of $p$, meaning $1/p + 1/p' = 1$. 

\begin{thm}\label{thm::opNorm}
Suppose $n \ge 2$, and $p \in [2,\infty)$. Then\\
(i) 
\begin{equation*}
\bb{E} \| X \|_{p',p} \lesssim \sqrt{p}\, n^{1/p}.
\end{equation*}
(ii) With probability $1 - O(n^{-4})$,
\begin{equation*}
| \| X \|_{p',p} - \bb{E} \| X \|_{p',p}  | \lesssim \log n.
\end{equation*}
As a consequence, $\| X \|_{p',p} \lesssim \sqrt{p \log n}\, n^{1/p}$ holds with probability $1 - O(n^{-4})$.
\end{thm}

An easy consequence of this theorem is a vector norm bound as follows.

\begin{cor}\label{cor::conctrNorm}
Suppose $p \in [2, \infty)$, and $\xi \in \bb{C}^n$ is a random vector with independent entries. Assume $\| \Re(\xi_j) \|_{\psi_2} \le 1,  \| \Im(\xi_j) \|_{\psi_2} \le 1$ for all $j \le n$. Then, with probability $1 - O(n^{-4})$, 
\begin{equation*}
\| \xi \|_p \lesssim \sqrt{p \log n}\, n^{1/p}.
\end{equation*}
\end{cor}

\subsection{Eigenvalue Perturbation}\label{sec::eigvalPtb}

Suppose $n \ge 2$ and $\mu_1, \mu_2 , \ldots , \mu_n$ are all positive numbers.  Let us denote $D_\mu = \diag( \mu_1,\ldots, \mu_n )$, i.e., a diagonal matrix consisting of $\mu_1,\ldots \mu_n$ on its diagonal. We will use notation $\ddiag(D_\mu^{-1})$ to extract the diagonal of $D_\mu^{-1}$ into an $n$-dimensional vector. 

Our goal is to establish $X \preceq D_\mu$ with high probability under some condition on $\{ \mu_j\}_{j=1}^n$. Apparently, when $D_\mu$ is a multiple of identity matrix (i.e., $\mu_1 = \ldots = \mu_n$), $X \preceq D_\mu$ is equivalent to $\lambda_{\max}(X) \le \mu_1$, and this problem has been studied in a wealth of papers (e.g., \citep{RudVer10,Ver10}). Allowing $D_{\mu}$ to be inhomogeneous (non-identical $\mu_j$'s), our quest can be regarded as a generalization of bounding the largest eigenvalue of a random matrix. 

Suppose $c_0' > 0$ is a suitably small absolute constant. We need the following condition on $\{ \mu_j\}_{j=1}^n$ that controls the fluctuation of $X$. This condition is almost identical to Assumption \ref{assmp::lambda} except that $d(\lambda)$ is replaced by $\ddiag(D_{\mu}^{-1})$. 
\begin{assmp}\label{assmp::mu}
Suppose either of the following: (1) there exists $p \in [2, \infty)$ such that  
\begin{equation}\label{ineqn::eigvalCon-2}
\sqrt{p\log n}\, n^{1/p}\| \ddiag(D_{\mu}^{-1}) \|_{p/(p-2)} \le c_0', 
\end{equation}
(2)
\begin{equation}\label{ineqn::eigvalCon2-2}
\log n \,\| \ddiag(D_{\mu}^{-1}) \|_{1} \le c_0'.
\end{equation}
\end{assmp}
Note that in fact (\ref{ineqn::eigvalCon2-2}) is a special case of (\ref{ineqn::eigvalCon-2}) up to some absolute constant (which we discussed in Section \ref{sec::main}). To see this, we set $p = 4\log n$ and use the fact that $n^{1/p} \asymp 1$, $\| u \|_{p/(p-2)} \asymp \| u \|_1$. The condition on $\| \ddiag(D_{\mu}^{-1}) \|_{p/(p-2)}$ aims to control the fluctuation of $X$ in an inhomogeneous way, where $p$ offers flexibility of choices that can be tailed to different settings of $\{ \mu_j \}_{j=1}^n$. With $p=2$, this condition encompasses the homogenous case (namely $\mu_j$'s are identical), though an additional $\sqrt{\log n}$ factor appears.

We will state our result for eigenvalue perturbation in a slightly more general form, which is useful in Section \ref{sec::match}. Suppose $\tau \in [0, 1]$ is any fixed number, and $g \in \bb{C}^n$ is a random vector whose entries have jointly independent subgaussian real parts and imaginary parts, with $\| \Re(g_j) \|_{\psi_2} \le 1, \| \Im(g_j) \|_{\psi_2} \le 1$ for any $j$, and $g$ is independent of $X$. 

\begin{thm}\label{thm::eigvalPtb}
Suppose $n \ge 2$. Under Assumption \ref{assmp::mu}, the following holds with probability $1 - O(n^{-4})$:
\begin{equation}\label{ineqn::eigval1}
z^* X z + \tau \| z \|_2 \Re(g^* z) \le z^* D_\mu z, \qquad \forall z \in \bb{C}^n. 
\end{equation}
In particular, 
\begin{equation*}
X \preceq D_{\mu}
\end{equation*}
holds with probability $1 - O(n^{-4})$. 
\end{thm}

\begin{rmk}
As we have seen in examples in Section \ref{sec::main}, when $p$ is large, the smallest element of $\{ \mu_j \}_{j=1}^n$ can be as small as $O(\polylog)$, which is vanishing compared to $\| X \|_2$. For example, if $\mu_j = C(n+1-j)\log^3 n$ where $C$ is a large constant, then (\ref{ineqn::eigvalCon2-2}) is satisfied.

Theorem \ref{thm::eigvalPtb} also generalizes Weyl's inequality in terms of top eigenvalues in the random setting. To see this point, suppose $\lambda_1 \ge \ldots \ge \lambda_n$, and $\mu_j = \lambda_1 - \lambda_{j+1}$. Rearranging $X \preceq D_{\mu}$, we deduce that the top eigenvalue of $\diag(\lambda_2,\ldots,\lambda_n) + X$ is no greater than $\lambda_1$. This is more general than Weyl's inequality in the sense that, with high probability, the perturbation $X$ shifts the top eigenvalue upward by no more than $\lambda_1 - \lambda_2$, which can be much smaller than $\| X \|_2$.
\end{rmk}
 
\vspace{2mm}

The proof of Theorem \ref{thm::eigvalPtb} is based on the chaining argument (\cite{dudley1967sizes}; see \cite{Van14} for a nice treatment). It turns out that we need to bound the covering number of an ellipsoid because $\mu_1,\ldots,\mu_n$ are not identical. Let $E_a$ be an ellipsoid in $\bb{R}^m$ with the lengths of axes$a_1,a_2,\ldots,a_m$, and $N(E_a)$ be the minimum number of unit balls that cover $E_a$. The following bound is from \cite{Dum04} and \cite{Dum06} (also see \cite{dumer2007covering}).

\begin{lem}\label{lem::covering}
For any $\theta \in (0,1/2)$, the covering number of any ellipsoid $E_a$ satisfies
\begin{equation}
\log N(E_a) \le |J| \log h(a) + |J_{\theta}| \log (\bar{C}/\theta),
\end{equation}
where $\bar{C}$ is a constant, $J = \{j: a_j > 1\}, J_{\theta} = \{j: a_j^2 \ge 1- \theta \}$ and $h(a) = \prod_{j \in J}a_j^{1/|J|}$.
\end{lem}

We will pick $\theta = 1/4$ in the proof of Theorem \ref{thm::eigvalPtb}. The above lemma says $\log N(E_a)$ is upper bounded by $|J| \log h(a) + O(|J_{\theta}|)$. \cite{Dum04} also provides a lower bound that shows $\log N(E_a) \ge  |J| \log h(a)$. The extra term $O(|J_{\theta}|)$, as is shown in the proof, is easy to control once we know how to bound $ |J| \log h(a)$.

\section{Eigenvector and Quadradic Equations}
Recall the spectral decomposition of $A$:
\begin{equation*}
A = \sum_{j=1}^n \lambda_j u_j u_j^*,
\end{equation*}
where $\lambda_1 \ge \lambda_2 \ge \ldots \ge \lambda_n$. To ease notations, henceforth we will drop the subscript and simply write $u$ (or $\tilde{u}$) to denote the leading eigenvectors. Let us denote $U = (u, U_\bot)$, where $U_\bot = (u_2,\ldots,u_n)$ is a $n \times (n-1)$ matrix. We are interested in the leading eigenvector $\tilde{u}$ of the perturbed Hermitian matrix $\tilde{A} = A + E$. Denote
\begin{equation} \label{eqn::EblockMat}
\tilde{E} = (u, U_\bot)^* E (u, U_\bot) = \left( \begin{array}{cc}
u^* E u & u^* E U_\bot \\ U_\bot^* E u & U_\bot^* E U_\bot \end{array} \right) =: 
\left( \begin{array}{cc}
E_{11} & E_{12} \\ E_{21} & E_{22} \end{array} \right).
\end{equation}

Let us also denote
\begin{align}
\tilde{u} &= (u + U_\bot q ) ( 1 + \| q \|_2^2)^{-1/2}, \label{eqn::tildeu}\\
\tilde{U}_\bot &= (U_\bot - u q^* ) ( I_{n-1} + qq^*)^{-1/2}, \label{eqn::tildeuBot}
\end{align}
where $q \in \bb{C}^{n-1}$. It is easy to check that $(\tilde{u}, \tilde{U}_\bot)$ forms an orthogonal basis. Our task is to study the $q$ for which $\tilde{u}$ is the leading eigenvector of $\tilde{A}$. 

If $\tilde{u}$ is an eigenvector of $\tilde{A}$, then we can find other orthogonal $n-1$ eigenvectors in the space spanned by $\tilde{U}_\bot$. Thus, we expect the necessary condition $\tilde{U}_\bot^* \tilde{A} \tilde{u} = 0$. It is also a sufficient condition, as shown in the next theorem. This leads to a system of quadratic equations (\ref{eqn::quad}).

\begin{thm}\label{thm::quadEqn}
Suppose $q \in \bb{C}^{n-1}$ satisfies
\begin{equation}\label{eqn::quad}
\cal{L} q = E_{21} - qE_{12}q,
\end{equation}
where
\begin{equation*}
\cal{L} := (\lambda_1 + E_{11}) I_{n-1} - ( \diag(\lambda_2,\ldots, \lambda_n) + E_{22} ).
\end{equation*}
Then $\tilde{u}$, defined in (\ref{eqn::tildeu}), is an eigenvector of $\tilde{A}$.
\end{thm}

\begin{proof}
It can be checked that (\ref{eqn::quad}) is equivalent to 
\begin{equation*}
(U_\bot^* - qu^*) ( A + E) (u + U_\bot q) = 0,
\end{equation*}
or simply $\tilde{U}_\bot^* \tilde{A} \tilde{u} = 0$. Let $R \in \bb{C}^{(n-1) \times (n-1)}$ be a unitary matrix that diagonalizes $\tilde{U}_\bot^* \tilde{A}\tilde{U}_\bot$. Then it is clear that $(\tilde{u}, \tilde{U}_\bot R)$ diagonalizes $\tilde{A}$, and thus $\tilde{u}$ is an eigenvector of $\tilde{A}$.
\end{proof}
This theorem does not associate $\tilde{u}$ with the largest eigenvalue $\lambda_{\max}(\tilde{A})$ (or any particular eigenvalue), and therefore we have to match $\tilde{u}$ with $\lambda_{\max}(\tilde{A})$ with some additional efforts (see Section \ref{sec::match}). Differing from methods that analyze $\tilde{u}$ from its variational definition (e.g., \cite{Vu11}), our analysis approaches $\tilde{u}$ from equations, which was adopted in the proof of Davis-Kahan $\sin \Theta$ theorem \citep{DavKah70}. The use of Davis-Kahan theorem is commonly aided with Weyl's inequality $| \lambda_j(\tilde{A}) - \lambda_j(A)| \le \| E \|_2$, which controls the fluctuation of the top eigenvalue and helps match eigenvectors with eigenvalues. Thanks to Theorem \ref{thm::eigvalPtb}, which serves as a random counterpart of Weyl's inequality, we are able to match $\tilde{u}$ with $\lambda_{\max}(\tilde{A})$.

\section{Properties of $\cal{L}^{-1}$}\label{sec::L-1}

The random matrix $\cal{L} \in \bb{C}^{(n-1) \times (n-1)}$ is a linear operator defined in Theorem \ref{thm::quadEqn}---see (\ref{eqn::quad}). The properties of $\cal{L}^{-1}$ are important because, if the quadratic equations (\ref{eqn::quad}) are approximately linear, then $q$ is roughly given by $q \approx \cal{L}^{-1} E_{21}$. We will justify this heuristics in Section \ref{sec::quad}.

We begin with a few notational definitions. Let 
\begin{equation*}
D_\lambda = \diag(\lambda_1 - \lambda_2, \lambda_1  - \lambda_3, \ldots, \lambda_1   - \lambda_n), \qquad D = D_\lambda + E_{11} I_{n-1},
\end{equation*}
be diagonal matrices of size $n-1$. For $m = 1,\ldots, n-1$, let $E_{22}^{(m)}$ be a copy of $E_{22}$ except that the $m$'th row and $m$'th column is set to zero, and denote $\Delta E_{22}^{(m)} = E_{22} - E_{22}^{(m)}$. Similarly we can define $\cal{L}^{(m)}$:
\begin{equation*}
\cal{L} := D - E_{22}, \qquad \cal{L}^{(m)} := D - E_{22}^{(m)}.
\end{equation*}

It is easy to see that $D \succeq D_\lambda/2$ under Assumption \ref{assmp::lambda} (where $c_0$ is small enough), since $\delta \gg E_{11}$ with high probability. A rigorous statement is given in Lemma \ref{lem::indct}. By Theorem \ref{thm::eigvalPtb}, $\cal{L}$ is invertible with probability $1-O(n^{-4})$ (for small $c_0$). A more useful result is the following lemma, which is a consequence of contraction mapping theorem. Its condition holds with high probability (see Lemma \ref{lem::indct}), so we do not need to worry about invertibility of $\cal{L}$ (and also $\cal{L}^{(m)}$).

\begin{lem}\label{lem::L-1norm}
Suppose $n \ge 2$, $D$ is invertible, and $\| E_{22} D^{-1} \|_{p,p} \le 1/2$. Then, $\cal{L}$ is invertible, and for any $u \in \bb{C}^{n-1}$, we have 
\begin{equation*}
\| D \cal{L}^{-1} u \|_p \le 2\|u\|_p.
\end{equation*}
The same conclusion holds if $E_{22}$ is replaced by $E_{22}^{(m)}$, and $\cal{L}$ by $\cal{L}^{(m)}$.
\end{lem}

Lemma \ref{lem::L-1norm} is not sufficient for our analysis, due to its deterministic nature. Suppose $E_{21}$ is a standard Gaussian vector and we set $u = E_{21}$, this lemma will lead to $\| D \cal{L}^{-1} E_{21} \|_p \le 2\|E_{21} \|_p \asymp n^{1/p}$ for any fixed $p$. There is no obvious way to convert the $\ell^p$ norm into $\ell^\infty$ norm while achieving the optimal bound $\| D \cal{L}^{-1} E_{21} \|_\infty = O(\sqrt{\log n})$ (with high probability).

Our solution is to look at $\cal{L}^{-1}$ from an iterative perspective. Our first useful observation is that, for any $y \in \bb{C}^{n-1}$, $x = D\cal{L}^{-1} y$ is the fixed point of a linear operator $\cal{T}$, where $\cal{T}x := E_{22}D^{-1} x + y$. In fact, $x = D\cal{L}^{-1} y \Leftrightarrow \cal{L} D^{-1}x = y \Leftrightarrow x = E_{22}D^{-1}x + y$, which is exactly $x = \cal{T}x$. This is also the underpinning of \textit{Jacobi iterative method} for numerically computing inverse of matrices.

This observation motivates us to study $D\cal{L}^{-1}y$ by looking at the sequence $\{ x^t \}$ generated by the iteration $x^{t+1} = \cal{T} x^t$. This algorithmic thinking, however, soon encounters a difficulty: we don't know how $x^t$ depends on $E_{22}D^{-1}$, which are both random. Analysis would be much easier if we could pretend that $E_{22}D^{-1}$ and $x^t$ are independent. 

It turns out that the dependence is so weak as if $E_{22}D^{-1}$ and $x^t$ were independent. To justify it, we need to track $n$ sequences closely related to $\{x^t\}$. 
For $m = 1,\ldots, n-1$ and $t \ge 1$, given (possibly non-deterministic) $v^0, v^{0,m} \in \bb{C}^{n-1}$, let 
\begin{align}
v^{t+1} &= E_{22}D^{-1} v^t + v^0, \label{eqn::vtIter} \\
v^{t+1,m} &= E_{22}^{(m)}D^{-1} v^{t,m} + v^{0,m}. \label{eqn::vtmIter}
\end{align}
In this paper, the superscript $m$ over a vector or a matrix signifies that the random quantity (e.g., $v^{t,m}, E_{22}^{(m)}$) is independent of the $m$'th row and $m$'th column of $E_{22}$. Clearly, from the iterative definition, if $v^{0,m}$ is independent of the $m$'th row and $m$'th column of $E_{22}$, so does $v^{t,m}$ for all $t\ge 1$.

The key insight is that we can rewrite (\ref{eqn::vtIter}) into
\begin{equation*}
v^{t+1}_m = [E_{22}D^{-1} v^{t,m}]_m +  [E_{22}D^{-1} (v^t - v^{t,m})]_m + v^0_m.
\end{equation*}
where the subscripts mean the $m$'th coordinates. On the right-hand side, the first term is amenable for analysis because of independence. The second term concerns that proximity of $v^t$ and $v^{t,m}$, whose difference under $\ell^p$ norm, crucially, does not accumulate as $t \to \infty$. To understand this, we need this following lemma.

\begin{lem}\label{lem::indct}
(i) With probability $1 - O(n^{-4})$,
\begin{equation*}
\| \tilde{E} \|_{\infty} \le C_0 \sqrt{\log n},
\end{equation*}
where $C_0>0$ is an absolute constant. 

(ii) Under Assumption \ref{assmp::lambda}, for $p \in [2,\infty]$, with probability $1 - O(n^{-2})$, the following holds: $D_{jj} \ge [D_\lambda]_{jj}/2$ and thus $\| D^{-1} \|_{p,p'} \le 2 \| D_\lambda^{-1} \|_{p,p'}$; and furthermore, for any $m=1,\ldots n-1$,
\begin{equation}
\| E_{22} D^{-1} \|_{p,p} \le 1/2, \quad \| E_{22}^{(m)} D^{-1} \|_{p,p} \le 1/2, \quad \| D^{-1} E_{21} \|_{p'} \le 1/2, \label{ineqn::lemIndct1} 
\end{equation}
where $p'$ satisfies $1/p + 1/p' = 1$.

(iii) For any $m = 1,\ldots,n-1$, if a random vector $\xi \in \bb{C}^{n-1}$ is independent of $m$'th row and $m$'th column of $E_{22}$, then under Assumption \ref{assmp::lambda}, for any $p \in [2,\infty]$, with probability $1 - O(n^{-4})$, 
\begin{equation}\label{ineqn::sparse}
\| \Delta E_{22}^{(m)} D^{-1} \xi \|_p \le \| \xi \|_{\infty} / 12.
\end{equation}
\end{lem}

This first part (i) is rather evident from the subgaussian assumption. Part (ii) is crucial in order to argue $\|v^t - v^{t,m} \|_p$ does not accumulate. Part (iii) utilizes independence, and leads to good coordinate-wise control of $v^t$ and $v^{t,m}$. 

Now we can begin our analysis for the iterative sequences. Fix $T := 5 + n$. Let $\cal{A}:= \cal{A}_1 \cap \cal{A}_2 \cap \cal{A}_3$, where
\begin{align}
\cal{A}_1 &:= \{ \| \tilde{E} \|_{\infty} \le C_0 \sqrt{\log n} \}, \label{eqn::A1}\\
\cal{A}_2 &:= \{ \| E_{22} D^{-1} \|_{p,p} \le 1/2 \} \cap \{ \| E_{22}^{(m)} D^{-1} \|_{p,p} \le 1/2, ~ \forall m \}, \label{eqn::A2}\\
\cal{A}_3 &:=\{ \| \Delta E_{22}^{(m)} D^{-1} v^{t,m} \|_p \le \| v^{t,m} \|_\infty / 12, ~ \forall 0 \le t \le T-1, ~ \forall m \} \label{eqn::A3}.
\end{align}
We will choose $v^{0,m}$, for each $m$, in a way such that it is independent of $m$'th row and $m$'th column of $E_{22}$ (see Section \ref{sec::quad}). According to the definitions, so are $v^{t,m}$. By Lemma \ref{lem::indct} and the union bound, the event $\cal{A}$
has probability $1 - O(n^{-2})$. Under the event $\cal{A}$, which holds high probability, bad examples of $E$ are excluded. The next theorem is deterministic in nature, because results in (i)-(iii) hold pointwise in $\cal{A} \cap \cal{B}$.

\begin{thm}\label{thm::indctL}
Suppose $n \ge 2$ and Assumption \ref{assmp::lambda}. Let $\omega > 0$ be a fixed positive number, and $\cal{B}$ be an event determined by
\begin{equation*}
\| v^0 - v^{0,m} \|_{p} \le \omega \sqrt{\log n}, \quad \| v^0 \|_{\infty} \le 3\omega \sqrt{\log n}, \quad \| v^{0,m} \|_{\infty} \le 3\omega \sqrt{\log n}.
\end{equation*}
Under $\cal{A} \cap \cal{B}$, the following inequalities hold: \\
(i) for $t = 0,\ldots T$,
\begin{align}
&\| v^t - v^{t,m} \|_{p} \le \omega' \sqrt{\log n}, \label{ineqn::vtDiffpNorm}\\
& \| v^t \|_{\infty} \le 3\omega' \sqrt{\log n}, \label{ineqn::vtInf}\\
 &\| v^{t,m} \|_{\infty} \le 3\omega' \sqrt{\log n}, \label{ineqn::vtmInf}
\end{align}
where $\omega' = 4\omega$. \\
(ii) The sequences $\{ v^t \}$ and $\{ v^{t,m} \}$ have limits in $\bb{C}^{n-1}$. Let $v^\infty = \lim_{t \to \infty} v^t$ and $v^{\infty,m} = \lim_{t \to \infty} v^{t,m}$. For $t=1,2,\ldots$ and $t = \infty$, inequalities (\ref{ineqn::vtDiffpNorm}) - (\ref{ineqn::vtmInf}) hold with $\omega'$ being replaced by $\omega'' := 8\omega$.\\
(iii) As a consequence, $D^{-1}v^{\infty}$ and $D^{-1}v^{\infty,m}$ are, respectively, a solution to $\cal{L}x = v^0$ and $\cal{L}^{(m)}x = v^{0,m}$.
\end{thm}

This theorem establishes sharp $\ell^\infty$ bounds on rescaled $\cal{L}^{-1}v^0$ and $\cal{L}^{-1}v^{0,m}$. The following easy corollary is obtained by setting $v^0 = v^{0,m} = E_{21}$, and it is proved in the proof of Theorem \ref{thm::indctQs}.

\begin{cor}\label{cor::L-1E21}
With probability $1 - O(n^{-2})$, $\| D \cal{L}^{-1} E_{21} \|_{\infty} = O( \sqrt{\log n} )$.
\end{cor}

This corollary says if we could ignore the quadratic term in (\ref{eqn::quad}) and solve an approximate linear equation $\cal{L} q \approx E_{21}$, the solution would already have the bound $O(\sqrt{\log n}/(\lambda_1 - \lambda_j))$ for each coordinate. The next section justifies this approximation.

\begin{proof}[Proof of Theorem \ref{thm::indctL}]
(i) From the iterative definitions, we have
\begin{align}
v^{t+1} - v^{t+1,m} &= \Delta E_{22}^{(m)} D^{-1} v^{t,m} + E_{22} D^{-1} (v^t - v^{t,m}) + (v^0 - v^{0,m}), \\
v^{t+1} &= E_{22}D^{-1}v^{t,m} + E_{22}D^{-1}(v^t - v^{t,m}) + v^0,
\end{align}
for any $m=1,\ldots n-1$. We now begin our induction on $t$. For $t=0$, the bounds (\ref{ineqn::vtDiffpNorm}) -- (\ref{ineqn::vtmInf}) are trivially true under $\cal{B}$. Suppose that we have proved (\ref{ineqn::vtDiffpNorm}) -- (\ref{ineqn::vtmInf}) for $t$. To deal with $t+1$, we notice that under $\cal{A}$, 
\begin{align*}
\| v^{t+1} - v^{t+1,m} \|_p & \le \| v^{t,m} \|_{\infty} / 12 + \| v^t - v^{t, m} \|_p / 2 + \| v^0 - v^{0,m} \|_p \\
&\le 3\omega' \sqrt{\log n}  / 12+ \omega' \sqrt{\log n}/2 + \omega \sqrt{\log n} \\
&= \omega' \sqrt{\log n}.
\end{align*}
For any $m=1,\ldots,n-1$, we can bound the $m$'th coordinate of $v^{t+1}$ as follows:
\begin{align*}
| v^{t+1}_m | &\le | [ \Delta E_{22}^{(m)} D^{-1} v^{t,m}]_m | + \| E_{22} D^{-1} ( v^t - v^{t,m} ) \|_{\infty} + \| v^0 \|_{\infty} \\
&\le \| \Delta E_{22}^{(m)} D^{-1} v^{t,m} \|_p + \| E_{22} D^{-1} ( v^t - v^{t,m} ) \|_p + \| v^0 \|_{\infty} \\
&\le \| v^{t,m} \|_{\infty}/12 + \| v^t - v^{t,m} \|_p /2 + \| v^0 \|_{\infty} \\
&< 2\omega' \sqrt{\log n},
\end{align*}
where, in the first inequality, we used the identity
\begin{equation*}
 [E_{22} D^{-1} v^{t,m}]_m  =   ((E_{22})_{m \cdot})D^{-1} v^{t,m} = [\Delta E_{22}^{(m)} D^{-1} v^{t,m}]_m 
\end{equation*} 
where $(E_{22})_{m \cdot}$ denotes the $m$'th row of $E_{22}$. 
This proves (\ref{ineqn::vtInf}) for the case $t+1$. The bound on $v^{t+1,m}$ follows easily, since 
\begin{equation*}
\| v^{t+1,m} \|_\infty \le \|v^t \|_\infty + \|v^t - v^{t,m} \|_\infty \le 3w' \sqrt{\log n}.
\end{equation*}
By induction, (\ref{ineqn::vtDiffpNorm}) -- (\ref{ineqn::vtmInf}) are true for $t \le T$ (we cannot prove the bounds for $t>T$ due to the definition of $\cal{A}_3$).

(ii) For any $t \ge 1$, 
\begin{equation*}
v^{t+1} - v^t = E_{22} D^{-1} (v^t - v^{t-1}).
\end{equation*}
Under $\cal{A}$, we have
\begin{equation*}
\| v^{t+1} - v^t \|_p \le \| E_{22} D^{-1} \|_{p,p} \| v^t - v^{t-1} \|_p \le  \| v^t - v^{t-1} \|_p /2.
\end{equation*}
This implies that $\{v^t\}$ is a Cauchy sequence, and therefore has a limit $v^{\infty} = \lim_{t \to \infty} v^t$ in $\bb{C}^{n-1}$. The same argument applies to $v^{t,m}$, and we let $v^{\infty, m}$ denote their respective limits. Moreover, for any integer $t > T$, $\| v^{t} - v^T \|_p \le \sum_{k=T}^{t-1} \| v^{k+1} - v^k \|_p \le \sum_{k=T}^{t-1} 2^{k-T} \| v^{T+1} - v^T \|_p$, and similar bounds hold for $ \| v^{t,m} - v^{T,m} \|_p$. This leads to
\begin{align*}
\| v^t - v^{t,m} \|_{p} &\le \| v^T - v^{T,m} \|_{p} + \| v^{t} - v^T \|_p + \| v^{t,m} - v^{T,m} \|_p \\
&\le \omega' \sqrt{\log n} + 2 ( \| v^{T+1} - v^T \|_p + \| v^{T+1,m} - v^{T,m} \|_p) \\
&\le \omega' \sqrt{\log n} + 2^{1-T}n^{1/p} \cdot ( \| v^{1} - v^0 \|_\infty + \| v^{1,m} - v^{0,m} \|_\infty ) \\
&\le \omega' \sqrt{\log n} + 2^{1-T}n^{1/p} \cdot 12\omega'  \sqrt{\log n}.
\end{align*}
where we used the trivial bound $\| x \|_p \le n^{1/p} \|x \|_{\infty}$ to convert the norms (note $n^{1/p}=1$ when $p=\infty$). Since $12 \cdot 2^{1-T} n^{1/p} \le 24 \cdot 2^{-T} n < 1$ (recall $T = n+5$), it follows that $\| v^t - v^{t,m} \|_{p} \le \omega'' \sqrt{\log n}$. To bound $\| v^t \|_{\infty}$, for any integer $t > T$, notice
\begin{align*}
\| v^t \|_\infty &\le \| v^T \|_\infty + \| v^t - v^T \|_p \le 3\omega'\sqrt{\log n} + 2^{1-T} \|v^1 - v^0\|_p \\
&\le 3\omega'\sqrt{\log n} + 2^{1-T} n^{1/p} \cdot 6\omega' \sqrt{\log n} \\
&\le 3\omega'' \sqrt{\log n},
\end{align*}
where we used $2^{2-T}n^{1/p} \le 1$. A similar bound holds for $\| v^{t,m} \|_\infty$. By taking limits, the same three bounds hold for $t = \infty$. This finishes the proof of part (ii).

(iii) We let $t \to \infty$ in (\ref{eqn::vtIter}) and (\ref{eqn::vtmIter}), which leads to $v^{\infty} = E_{22} D^{-1} v^{\infty} + v^0$ and a similar identity for $v^{\infty, m}$. So $D^{-1}v^\infty$ is a solution to the equation $Dx = E_{22}x + v^0$, or equivalently $\cal{L}x = v^0$. Similarly, $D^{-1}v^{\infty,m}$ is a solution to $\cal{L}^{(m)}x = v^{0,m}$. This concludes the proof.
\end{proof}

\section{Analysis of Quadratic Equations}\label{sec::quad}
Recall that our goal is to study $q$, the solution to the system of quadratic equations:
\begin{equation}\tag{\ref{eqn::quad}}
\cal{L} q = E_{21} - qE_{12}q,
\end{equation}
If we could discard the quadratic term, then $q$ would be determined by a linear system and Corollary \ref{cor::L-1E21} would have achieved our goal. However, this is not the case, but the good news is that the quadratic equations are approximately linear.
The analysis is similar to what we have done in Section \ref{sec::L-1}. Let $q^0 = q^{0,m} = 0$, and define recursively
\begin{align*}
q^{s+1} &= \cal{L}^{-1} E_{21} + \cal{L}^{-1} q^s E_{12} q^s, \\
q^{s+1, m} &= (\cal{L}^{(m)})^{-1} E_{21} +  (\cal{L}^{(m)})^{-1} q^{s,m} E_{12} q^{s,m}.
\end{align*}
for $s = 0,1,\ldots$ For each $s$, we analyze $\cal{L}^{-1} q^s$ and $(\cal{L}^{(m)})^{-1} q^{s,m}$ via the sequences we produced in Section \ref{sec::L-1}. So there are two loops of iterations: $\{q^s\}, \{q^{s,m}\}$ are sequences produced by the outer iteration loop; and for any fixed $s \ge 1$, we begin an inner iteration loop by setting $v^0$ to $q^s$ and $v^{0,m}$ to $q^{s,m}$, and iterates according to (\ref{eqn::vtIter}) and (\ref{eqn::vtmIter}). Let $\cal{A}^s$ $(s \ge 1)$ be the intersection of $\cal{A}_1, \cal{A}_2$ and $\cal{A}_3$ as defined in (\ref{eqn::A1}) and (\ref{eqn::A2}), with initializers being $v^0 = q^s$ and $v^{0,m} = q^{s,m}$. For $s=0$, we set $v^0 = v^{0,m} = E_{21}$, and define $\cal{A}^0 = \cal{A}_1 \cap \cal{A}_2 \cap \cal{A}_3$ with initializer $E_{21}$. Note that the way $q^{s, m}$ is defined assures independence of $q^{s,m}$ and $m$'th row and $m$'th column of $E_{22}$. We also define events
\begin{align*}
\cal{B}^0 &:= \{  \| E_{21} \|_{\infty} \vee |E_{11}| \le C_0 \sqrt{\log n} \} \cap \{  \| D^{-1} E_{21} \|_{p'} \le 1/2 \},
 \\
\cal{B}^s &:= \{ \| q^s - q^{s,m} \|_{p} \le 4\kappa_1\delta^{-1} \sqrt{\log n}, ~ \| q^s \|_{\infty} \le 4\kappa_2\delta^{-1} \sqrt{\log n}, \\
&~~~~~~ \| q^{s,m} \|_{\infty} \le 4\kappa_2\delta^{-1} \sqrt{\log n} \}, \qquad \forall s \ge 1.
\end{align*}
Here $C_0>0$ is the same constant as in Lemma \ref{lem::indct} part (i), and $\kappa_1 = 8C_0/3, \kappa_2 = 3\kappa_1$. \\


The results in Section \ref{sec::L-1} serve as building blocks for the analysis here, as we will need bounds for $\cal{L}^{-1} q^s$ and $(\cal{L}^{(m)})^{-1} q^{s,m}$.
\begin{thm}\label{thm::indctQs}
Suppose $n\ge 2$ and Assumption \ref{assmp::lambda}. Under $\bigcap_{s=0}^{T-1}\cal{A}^s \cap \cal{B}^0$, we have \\
(i) for $s=1,\ldots,T$ and $m = 1,\ldots,n-1$, 
\begin{align}
&\| D (q^{s,m} - q^s) \|_p \le 2\kappa_1 \sqrt{ \log n}, \label{ineqn::qDiff}\\
&\| D q^s \|_\infty \le 2\kappa_2\sqrt{ \log n}, \label{ineqn::qt}\\
&\| D q^{s,m} \|_\infty \le 2\kappa_2\sqrt{ \log n}, \label{ineqn::qtm}
\end{align}
where $\kappa_1 = 8C_0/3, \kappa_2 = 3\kappa_1$ are absolute constants.\\
(ii) for all $s \ge 1$,
\begin{align}
\begin{cases}
\| D q^s \|_\infty \le 3 \kappa_2 \sqrt{ \log n}, \\
\| D (q^{s+1} - q^s) \|_p \le \| D (q^s - q^{s-1}) \|_p / 2.
\end{cases}
\end{align}
As a consequence, $\{ q^s \}_{s=1}^\infty$ converges to a limit $q^\infty$, which is a solution to
\begin{equation*}
\cal{L} q = E_{21} + q E_{12} q.
\end{equation*}
And $q^\infty$ satisfies $\| D_\lambda q^\infty \|_\infty \le 6\kappa_2 \sqrt{\log n}$. 
\end{thm}

Together with Theorem \ref{thm::quadEqn}, this theorem implies $\tilde{u} = (u+U_\bot q^\infty)( 1 + \| q \|_2^2)^{-1/2}$ is an eigenvector of $\tilde{A}$ and that 
\begin{equation*}
| \langle \tilde{u}, u_j \rangle | = ( 1 + \| q \|_2^2)^{-1/2}|q_{j-1}|  \le 6\kappa_2\sqrt{\log n}/(\lambda_1 - \lambda_j), \qquad \, \forall2 \le  j \le n.
\end{equation*}
with probability $1 - O(n^{-1})$. The remaining task, therefore, is to show that $\tilde{u}$ is indeed the leading eigenvector, which is proved in Section \ref{sec::match}.

\begin{proof}[Proof of Theorem \ref{thm::indctQs}]
We assume that the constant $c_0$ in Assumption \ref{assmp::lambda} is small enough such that $\delta \ge \max\{2C_0\sqrt{\log n}, 96\kappa_2n^{1/p}\sqrt{\log n} \}$. Under $\cal{B}^0$ and this assumption, each element along the diagonal of $D$ satisfies $D_{jj} \ge \delta /2$.\\

(i) We will prove the inequalities through induction on $s$. For $s=1$, observe that $q^1 = \cal{L}^{-1} E_{21}, q^{1,m} = (\cal{L}^{(m)})^{-1} E_{21}$, and thus we can invoke Theorem \ref{thm::indctL}, in which we set initializers $v^0 = v^{0,m} = E_{21}$ and $\omega = C_0/3$. By Theorem \ref{thm::indctL}, under $\cal{A}^0 \cap \cal{B}^0$, the results (i)-(iii) apply, so
\begin{equation}\label{ineqn::indctQuadBase}
\| D (q^{1,m} - q^1) \|_p \le \kappa_1 \sqrt{ \log n}, ~ \| D q^1 \|_\infty \le \kappa_2\sqrt{ \log n}, ~ \| D q^{1,m} \|_\infty \le \kappa_2\sqrt{ \log n},
\end{equation}
This proves inequalities (\ref{ineqn::qDiff}) - (\ref{ineqn::qtm}) for the base case $s=1$, which also implies Corollary \ref{cor::L-1E21} since $\bb{P}(\cal{A}^0 \cap \cal{B}^0) = 1 - O(n^{-2})$. \\

To proceed with the inductive step, suppose (\ref{ineqn::qDiff}) -- (\ref{ineqn::qtm}) are true for $s$, and we will consider the case $s+1$. First we bound $\| D(q^{s+1,m} - q^{s+1})\|_p$. 
\begin{align}
& ~~~~ \| D(q^{s+1,m} - q^{s+1}) \|_p \notag \\
&\le \| D\cal{L}^{-1}E_{21} - D(\cal{L}^{(m)})^{-1} E_{21} \|_p + \| ( D\cal{L}^{-1} q^s - D(\cal{L}^{(m)})^{-1} q^{s,m} ) E_{12} q^s \|_p  \notag \\  
&  ~~~~ + \| D(\cal{L}^{(m)})^{-1} q^{s,m} ( E_{12} q^{s,m} - E_{12}q^s ) \|_p. \label{ineqn::qbound1}
\end{align}
The first term in the right-hand side is just $\| D(q^1 - q^{1,m}) \|_p$, which is bounded by $\kappa_1 \sqrt{\log n}$. To bound the second term and third term, note that $\| x \|_p \le 2\delta^{-1}\| D x \|_p, \| x \|_\infty \le 2\delta^{-1} \| Dx \|_\infty$, so (\ref{ineqn::qDiff}) -- (\ref{ineqn::qtm}) imply 
\begin{equation}\label{ineqn::q1}
\| q^s - q^{s,m} \|_{p} \le 4\kappa_1 \delta^{-1} \sqrt{\log n}, ~ \| q^s \|_{\infty} \le 4\kappa_2 \delta^{-1} \sqrt{\log n}, ~ \| q^{s,m} \|_{\infty} \le 4\kappa_2 \delta^{-1} \sqrt{\log n}.
\end{equation}
which is exactly the event $\cal{B}^s$. Now we resort to Theorem \ref{thm::indctL}, in which we set $v^0 = q^s$, $v^{0,m} = q^{s,m}$, and $\omega = 4\kappa_1\delta^{-1}$. According to Theorem \ref{thm::indctL}, under $\cal{A}^s\cap \cal{B}^s$,
\begin{align}
&\| D\cal{L}^{-1} q^s - D(\cal{L}^{(m)})^{-1} q^{s,m} \|_p \le 32\kappa_1\delta^{-1} \sqrt{\log n}, \label{ineqn::qbound2}\\
&\| D\cal{L}^{-1} q^{s} \|_{\infty} \le 32\kappa_2\delta^{-1} \sqrt{\log n}, \label{ineqn::qbound3}\\
&\| D(\cal{L}^{(m)})^{-1} q^{s,m} \|_{\infty} \le 32\kappa_2\delta^{-1} \sqrt{\log n}. \label{ineqn::qbound4}
\end{align}
Moreover, by H\"{o}lder's inequality, under $\cal{B}^0$,
\begin{align}
&| E_{12} q^s | \le \| D^{-1} E_{21} \|_{p'} \| D q^s \|_{p} \le n^{1/p} \| D q^s \|_\infty / 2, \label{ineqn::qbound5}\\
& | E_{12} (q^{s,m} - q^s) | \le \|  D^{-1}E_{21} \|_{p'} \| D (q^{s,m} - q^s) \|_{p} \le  \| D (q^{s,m} - q^s) \|_{p} /2, \notag
\end{align}
where we used the trivial bound $\| x \|_p \le n^{1/p} \| x \|_\infty$ (note $n^{1/p}=1$ when $p=\infty$). These inequalities, together with (\ref{ineqn::qDiff}), (\ref{ineqn::qt}), (\ref{ineqn::qbound2}), (\ref{ineqn::qbound4}), lead to a bound on (\ref{ineqn::qbound1}):
\begin{align*}
\| D(q^{s+1,m} - q^{s+1}) \|_p &\le \kappa_1 \sqrt{\log n} + 32\kappa_1\delta^{-1} \sqrt{\log n}\cdot n^{1/p}\kappa_2\sqrt{ \log n} \\
& ~~~~ +32\kappa_2\delta^{-1} n^{1/p}\sqrt{\log n} \cdot \kappa_1 \sqrt{ \log n} \\
& \le 2\kappa_1 \sqrt{\log n},
\end{align*}
since $64\kappa_2n^{1/p} \sqrt{\log n} \le \delta$. This proves (\ref{ineqn::qDiff}) for the case $s+1$. 

To prove (\ref{ineqn::qt}) for the case $s+1$, note
\begin{align*}
\| Dq^{s+1} \|_\infty &\le \| D \cal{L}^{-1} E_{21} \|_{\infty} + \| \cal{D} \cal{L}^{-1} q^s \|_{\infty} | E_{12} q^s |  \\
&\le  \| D q^1 \|_{\infty} + 32\kappa_2\delta^{-1} \sqrt{\log n} \cdot n^{1/p} \| Dq^s\|_{\infty}/2 \\
&\le \kappa_2  \sqrt{\log n} + 32\kappa_2\delta^{-1} \sqrt{\log n} \cdot n^{1/p} \kappa_2 \sqrt{\log n},
\end{align*}
where we used (\ref{ineqn::qbound3}) and (\ref{ineqn::qbound5}) in the second inequality, and (\ref{ineqn::indctQuadBase}) and (\ref{ineqn::qt}) in the third inequality. Since $32\kappa_2n^{1/p} \sqrt{\log n} \le \delta/3$, we derive $\| Dq^{s+1} \|_\infty \le 4\kappa_2\sqrt{\log n}/3$, which proves (\ref{ineqn::qt}) for the case $s+1$. Finally, 
\begin{align*}
\| Dq^{s+1,m} \|_\infty &\le \| Dq^{s+1} \|_\infty  + \| D(q^{s+1,m} - q^{s+1}) \|_p \le 2\kappa_1 \sqrt{\log n} + 4\kappa_2 \sqrt{\log n}/3 \\
&= 2\kappa_2 \sqrt{\log n}.
\end{align*}
This completes the inductive step. Therefore, under $\bigcap_{s=0}^{T-1}\cal{A}^s \cap \cal{B}^0$, inequalities (\ref{ineqn::qDiff}) - (\ref{ineqn::qtm}) are true for $s = 1,\ldots, T$ and any $m$.

(ii) Observe 
\begin{equation*}
\| D(q^{s+1} - q^s) \|_p \le \| D \cal{L}^{-1} (q^s - q^{s-1}) \|_p | E_{12}q^s| + \| D \cal{L}^{-1}q^{s-1}\|_p | E_{12} (q^s - q^{s-1}) |.
\end{equation*}
By Lemma \ref{lem::L-1norm}, we can bound $\| D \cal{L}^{-1} (q^s - q^{s-1}) \|_p$ by $2\|q^s - q^{s-1} \|_p$ and $\| D \cal{L}^{-1}q^{s-1}\|_p$ by $2 \|q^{s-1}\|_p$. Similar to the proof of part (i), we can use H\"{o}lder's inequality and $\| D^{-1} E_{21} \|_{p'} \le 1/2$ to bound $| E_{12}q^s|$ and $| E_{12} (q^s - q^{s-1}) |$. This results in
\begin{align}
\| D(q^{s+1} - q^s) \|_p &\le 2\| q^s - q^{s-1} \|_p  \cdot \| Dq^s \|_p/2 + 2\| q^{s-1} \|_p \cdot \|  D(q^s - q^{s-1}) \|_p/2 \notag\\
&\le 2\delta^{-1} n^{1/p} \|  D(q^s - q^{s-1}) \|_p (\| Dq^s \|_\infty + \| Dq^{s-1}\|_\infty), \label{ineqn::qbound5-2}
\end{align}
where, again, we used $\| x \|_p \le 2\delta^{-1} \|Dx\|_p$ and $\| x \|_p \le n^{1/p} \|x \|_\infty$. When $s\le T$, we can use what we proved in part (i) to derive $\| Dq^s \|_\infty + \| Dq^{s-1}\|_\infty \le 4\kappa_2\sqrt{\log n}$, so for $s=1,\ldots, T$, 
\begin{equation*}
\| D(q^{s+1} - q^s) \|_p \le 8\kappa_2\delta^{-1} n^{1/p}\sqrt{\log n}\,  \|  D(q^s - q^{s-1}) \|_p \le \|  D(q^s - q^{s-1}) \|_p/2.
\end{equation*}
This implies $\| D(q^{T+1} - q^T) \|_p \le 2^{-T} n^{1/p} \| Dq^1\|_\infty \le 2^{1-T} \kappa_2 n^{1/p} \sqrt{\log n}$. Since $2^{1-T} n^{1/p} \le 2^{1-T} n \le 1/12$, this yields $\| D(q^{T+1} - q^T) \|_p \le \kappa_2 \sqrt{\log n}/12$. Now we are going to show, via induction on $s$, that 
\begin{equation}
\begin{cases}
\| Dq^s \|_{\infty} \le 3\kappa_2 \sqrt{\log n}, \\
\| D(q^{s+1} - q^s) \|_p \le \|  D(q^s - q^{s-1}) \|_p/2,
\end{cases}
\forall s \ge 1. \label{ineqn::qbound6}
\end{equation}
From part (i) and the above argument, we know (\ref{ineqn::qbound6}) hold for $s \le T$. Suppose for $\ell \ge T$, the inequalities in (\ref{ineqn::qbound6}) are true for any $s \le \ell$; and let us consider case $\ell + 1$. Since $\| D(q^{s+1} - q^s) \|_p$ decays geometrically, we have $\| D(q^{s+1} - q^s) \|_p \le 2^{T-s} \| D(q^{T+1} - q^T) \|_p$ for $T \le s \le \ell$. Thus,
\begin{align*}
\| D q^{\ell+1} \|_\infty &\le \| D q^{T} \|_\infty + \sum_{s=T}^\ell \| D(q^{s+1} - q^s) \|_p \le  \| D q^{T} \|_\infty + \sum_{s=T}^\ell 2^{T-s} \| D(q^{T+1} - q^T) \|_p \\
&\le 2\kappa_2 \sqrt{\log n} + 2 \cdot \kappa_2 \sqrt{\log n}/12 < 3\kappa_2 \sqrt{\log n}.
\end{align*} 
Furthermore, in (\ref{ineqn::qbound5-2}) we set $s = \ell + 1$, and the desired inequality $\| D(q^{\ell+2} - q^{\ell+1}) \|_p \le \|  D(q^{\ell+1} - q^{\ell}) \|_p/2$ follows from $\| D q^{\ell+1} \|_\infty + \| D q^{\ell} \|_\infty \le 6\kappa_2\sqrt{\log n}$ and $24\kappa_2n^{1/p}\sqrt{\log n} \le \delta$. This completes the induction process and establishes (\ref{ineqn::qbound6}) for any $s \ge 1$. Apparently, (\ref{ineqn::qbound6}) implies that $q^s$ converges to a limit $q^\infty := \lim_{s \to \infty} q^s$ and $\| D q^\infty \|_\infty \le 3\kappa_2 \sqrt{\log n}$. Letting $s \to \infty$ in the recursive formula of $q^s$, we derive
\begin{equation*}
q^\infty = \cal{L}^{-1}E_{21} + \cal{L}^{-1} q^\infty E_{12} q^\infty.
\end{equation*}
It follows that
\begin{equation*}
\cal{L} q^\infty = E_{21} +  q^\infty E_{12} q^\infty.
\end{equation*}
Finally, we use $|E_{11}| \le \delta / 2$ to obtain $D_{jj} \ge (\lambda_1 - \lambda_j)/2 = [D_{\lambda}]_{jj}/2$, and therefore $\| D_\lambda q^\infty \|_\infty \le  2\| D q^\infty \|_\infty \le 6\kappa_2 \sqrt{\log n}$.
%

\end{proof}

Clearly the event $\bigcap_{s=0}^{T-1}\cal{A}^s \cap \cal{B}^0$ has probability $ 1 - O(n^{-1})$. The following corollary is a preparation for Section \ref{sec::match}.

\begin{cor}\label{cor::q2norm}
Suppose $n \ge 2$ and Assumption \ref{assmp::lambda} with sufficiently small constant $c_0$. With probability $1 - O(n^{-1})$, 
\begin{equation*}
\| \tilde{u} - u \|_2 \le \| q^\infty \|_2 \le 6\kappa_2\sqrt{\log n} \| \ddiag(D_\lambda^{-1}) \|_2.
\end{equation*}
As a consequence, when $c_0$ in Assumption \ref{assmp::lambda} is small enough, $\| q^\infty \|_2  \le 1/4$.
\end{cor}
\begin{proof}
By Theorem \ref{thm::indctQs}, $\| D_\lambda q^\infty \|_\infty \le 6 \kappa_2 \sqrt{\log n}$ holds with probability $1 - O(n^{-2})$. So
\begin{equation*}
\| q^\infty \|_2 \le \big( \sum_{j=1}^{n-1} \frac{ 1}{[D_\lambda]_{jj}^2} \big)^{1/2} \| D_\lambda q^\infty \|_\infty = \| \ddiag(D_\lambda^{-1}) \|_2 \| D_\lambda q^\infty \|_{\infty},
\end{equation*}
This leads to the desired bound on $\| q^\infty \|_2$. In light of Lemma \ref{lem::simpleNormBound} (also see (\ref{DlambdaConv})), under Assumption \ref{assmp::lambda} where constant $c_0$ is small enough, $\| q^\infty \|_2 \le 1/4$. It is easy to check that $(1 + a )^{-1/2} \ge 1 - a/2$ for any $a \in [0,1]$, so
\begin{equation*}
\| \tilde{u} - u \|_2^2 = 2 - 2\Re(\langle \tilde{u}, u \rangle) = 2 - 2(1 + \|q^\infty \|_2^2)^{-1/2} \le \| q^\infty \|_2^2.
\end{equation*}
This finishes the proof of the corollary.
\end{proof}

\section{Matching Top Eigenvalue with $\tilde{u}$}\label{sec::match}
In the previous section, we constructed a solution $q^\infty$ to the quadratic equations (\ref{eqn::quad}) via an iterative process, in which we maintained a sharp bound $\| Dq^s \|_\infty$ for each iteration $s$. From Theorem \ref{thm::quadEqn}, the associated vector
\begin{equation}\tag{\ref{eqn::tildeu}}
\tilde{u} = (u + U_\bot q^\infty )(1 + \| q^\infty \|_2^2)^{-1/2} 
\end{equation}
is an eigenvector of $\tilde{A}$. In this section, we will show that $\tilde{u}$ is indeed the \textit{leading} eigenvector of $\tilde{A}$.\\

To accomplish this, it suffices to show that the corresponding eigenvalue $\tilde{\lambda}$ is larger than other eigenvalues of $\tilde{A}$. Our strategy is to argue that (i) $\tilde{\lambda} > (\lambda_1 + \lambda_2)/2$; (ii) any other eigenvalue of $\tilde{A}$ is smaller than $(\lambda_1 + \lambda_2)/2$. The first claim (i) can be verified with the help of the bound we have established for $q^\infty$; and the second claim (ii) can be demonstrated by finding an upper bound on $\sup_{x: x \bot \tilde{u}} x^* \tilde{A} x$, which is achieved by using the eigenvalue perturbation result in Section \ref{sec::eigvalPtb}.

\begin{thm}\label{thm::match}
Under Assumption \ref{assmp::lambda}, with probability $1 - O(n^{-1})$, the leading eigenvalue of $\tilde{A}$ has multiplicity $1$, and $\tilde{u}$ is a leading eigenvector of $\tilde{A}$. 
\end{thm}

Together with Theorem \ref{thm::indctQs}, this proves the main result Theorem \ref{thm::main} in Section \ref{sec::main}.

\appendix
\section{Appendix}  
We provide additional proofs in the appendix. The notations of constants (e.g., $C_1,C_2$) will be reused in different proofs.
\subsection{Proof of Theorem \ref{thm::main2}}

\begin{proof}
Since $A = \diag\{ \lambda_1,\ldots, \lambda_n \}$, where $\lambda_1 > \lambda_2 \ge \ldots \ge \lambda_n$, clearly if we choose $u_j = e_j$ (which forms the standard basis in $\bb{R}^n$), the assumption on randomness holds. And because $E$ is a real symmetric matrix, we only need to consider real vectors. Denote $\tilde{u} = (a, w)^T$, where $a \in \bb{R}$ and $w \in \bb{R}^{n-1}$. Without loss of generality we assume $a \ge 0$. By definition, the leading eigenvector maximizes
\begin{equation*}
\tilde{u}^T \tilde{A} \tilde{u} = \lambda_1 a^2 + \sum_{j=1}^{n-1} \lambda_{j+1} w_j^2 + 2a \langle g, w \rangle = \lambda_1 + 2a \langle g, w \rangle -  \sum_{j=1}^{n-1} (\lambda_1 - \lambda_{j+1}) w_j^2,
\end{equation*}
where we used the normalization $a^2 + \| w \|_2^2 = 1$. Let $\rho_1,\ldots, \rho_{n-1}$ be positive numbers determined by equations
\begin{equation}\label{eqn::rho}
g_j^2 \rho_j^{-1} = \lambda_1- \lambda_{j+1} + \sum_{j =1 }^{n-1} \rho_j, \qquad \forall 1 \le j \le n-1.
\end{equation} 
The existence of $\rho_j$'s and their positivity will be verified later. Using AM-GM inequality,
\begin{align*}
\sum_{j=1}^{n-1} \big( 2a g_j w_j - (\lambda_1 - \lambda_{j+1}) w_j^2 \big) &\le \sum_{j=1}^{n-1} \big(  \rho_j a^2 + \rho_j^{-1} g_j^2 w_j^2 -  (\lambda_1 - \lambda_{j+1}) w_j^2 \big)\\
&= \gamma + \sum_{j=1}^{n-1} ( - \gamma + g_j^2 \rho_j^{-1}- (\lambda_1 - \lambda_{j+1}) ) w_j^2 \\
& = \gamma,
\end{align*}
where we denote $\gamma := \sum_{j =1 }^{n-1} \rho_j$. We claim that the maximum of $\tilde{u}^T \tilde{A} \tilde{u}$ is $\lambda_1 + \gamma$,  and its unique maximizer satisfies $\rho_j a = g_j w_j$ for all $j = 1,\ldots, n-1$. In fact, the equality holds if and only if $\sqrt{\rho_j}a = g_jw_j/\sqrt{\rho_j}$ for all $j$; and using $a^2 + \| w \|_2^2 = 1$ and $a \ge 0$, we can uniquely determine
\begin{equation*}
a = \big( 1 + \sum_{j=1}^{n-1} \rho_j^2 / g_j^2 \big)^{-1/2}, \quad w_j = \rho_j a / g_j, \qquad \forall j=1,\ldots, n-1.
\end{equation*}
The expression is valid almost surely (excluding $g_j = 0$). Note the amount of projection of $\tilde{u}$ onto $u_j$ is $\langle \tilde{u}, u_{j+1} \rangle = w_j$. To derive a lower bound on $|w_j|$, we need to study $\rho_j$. From (\ref{eqn::rho}), we have
\begin{equation*}
\gamma = \sum_{j=1}^{n-1}\frac{g_j^2}{\lambda_1 - \lambda_{j+1} + \gamma}, \qquad \rho_j = \frac{g_j^2}{\lambda_1 - \lambda_{j+1} + \gamma}, \quad \forall j=1,\ldots, n-1.
\end{equation*}
Here the first equation is obtained by summing the second one over $j$. It is clear that there exists a positive $\gamma$ that satisfies the first equation (almost surely); and the positivity of $\rho_j$ follows immediately from the second equation. Moreover, with probability $1 - O(n^{-2})$, $\max_j |g_j| \le \sqrt{6\log n}$, so $\gamma \le 6\log n \sum_j   (\lambda_1 - \lambda_{j+1})^{-1}$. Recall that $d(\lambda)$ is an $(n-1)$-dimensional vector with $j$'th entry being $(\lambda_1 - \lambda_{j+1})^{-1}$. Thus,
\begin{equation*}
\gamma \le 6 \log n \| d(\lambda) \|_1 \le 6 \log n \cdot n^{2/p} \| d(\lambda) \|_{p/(p-2)}, 
\end{equation*}
where we used H\"{o}lder's inequality. Under Assumption \ref{assmp::lambda}, $\delta \ge \| d(\lambda) \|_{p/(p-2)}^{-1} \ge c_0^{-1} \sqrt{\log n}\, n^{1/p}$, so
\begin{equation*}
\gamma \le 6 \sqrt{\log n} \, n^{1/p} \cdot \sqrt{\log n}\, n^{1/p} \| d(\lambda) \|_{p/(p-2)} \le 6 c_0^2 \delta.
\end{equation*}
When $c_0$ is small enough such that $6 c_0^2 \le 1$, we deduce $\gamma \le \delta$, and hence
\begin{equation*}
| w_j | = \frac{ |g_j| a }{\lambda_1 - \lambda_{j+1} + \gamma} \ge \frac{ |g_j| a }{2(\lambda_1 - \lambda_{j+1})}.
\end{equation*}
We claim that $a \ge 1/2$. Once we establish this claim, we will arrive at $|w_j| \ge |g_j|/4(\lambda_1 - \lambda_{j+1})$, thus finishing this proof. To prove this claim, observe that 
\begin{equation*}
\sum_{j=1}^{n-1} \rho_j^2 / g_j^2 \le \sum_{j=1}^{n-1} \frac{6\log n}{(\lambda_1 - \lambda_{j+1})^2} = 6 \log n \| d(\lambda) \|_2^2 \le 6 \log n \cdot n^{2/p} \| d(\lambda) \|_{p/(p-2)}^2,
\end{equation*}
where we used $\max_j |g_j| \le \sqrt{6\log n}$, and also $\| d(\lambda) \|_2 \le n^{1/p} \| d(\lambda) \|_{p/(p-2)}$ from Lemma \ref{lem::simpleNormBound}. By Assumption \ref{assmp::lambda}, we deduce
\begin{equation*}
\sum_{j=1}^{n-1} \rho_j^2 / g_j^2 \le 6c_0^2 \le 1,
\end{equation*}
and therefore $a \ge 1/\sqrt{2} > 1/2$. The proof is complete.
\end{proof}

\subsection{Proof of Theorem \ref{thm::main3}}
\begin{proof}
Denote $\tilde{A}_0 = A_0 + G$. In this proof, we will suppress the subscript $2$ in matrix operator norm $\| \cdot \|_2$. First notice that $\lambda_{\min}(A_0 + G) \ge \lambda_n - \| G \|$, which is nonnegative with probability $1 - O(e^{-cn})$ where $c$ is a constant, (see e.g., \cite{Ver16}). Thus, $ \lambda_{\max}(\tilde{A}_0) = \| \tilde{A}_0 \|$ and it suffices to prove $\| \tilde{A}_0 \| > \lambda_1$ with probability $1 - O(n^{-2})$.  It is easy to see that $\| \tilde{A}_0 \|$ concentrates well around its mean $\bb{E} \| \tilde{A}_0 \|$. Indeed, we view $ \| A_0 + G \|$ as a function of $n(n-1)/2$ i.i.d.\ standard Gaussian variables $\{G_{jk} \}_{1 \le j < k \le n-1}$ and $\{ G_{jj}/\sqrt{2} \}_{j=1}^n$, and we find that $\| A_0 + G \|$ is a $\sqrt{2}$-Lipschitz function. By Gaussian concentration inequality \citep{BouLugMas13}, for any $t \ge 0$,
\begin{equation}\label{ineqn::tailA0}
\bb{P} ( \| \tilde{A}_0 \| - \bb{E} \| \tilde{A}_0 \| > t ) \le e^{- t^2/4}.
\end{equation}
A lower tail bound also holds, which reads
\begin{equation}\label{ineqn::tailA0-2}
\bb{P} ( \| \tilde{A}_0 \| - \bb{E} \| \tilde{A}_0 \| < - t ) \le e^{- t^2/4}.
\end{equation}
We will use (\ref{ineqn::tailA0}) to show $(\bb{E} \| \tilde{A}_0 \| + 2)^2 \ge \bb{E} \| \tilde{A}_0 \|^2$, as follows:
\begin{align*}
\bb{E} \| \tilde{A}_0 \|^2 &= 2\int_0^\infty t\, \bb{P} ( \| \tilde{A}_0 \| > t ) \; dt \\
&\le 2 \int_0 ^{\bb{E} \| \tilde{A}_0 \|}  t \;dt + 2 \int_{\bb{E} \| \tilde{A}_0 \|}^\infty t \, \bb{P} ( \| \tilde{A}_0 \| > t ) \; dt \\
&\le (\bb{E} \| \tilde{A}_0 \| )^2 + 2 \int_0^\infty (\bb{E} \| \tilde{A}_0 \|  + t) \bb{P} ( \| \tilde{A}_0 \| > \bb{E} \| \tilde{A}_0 \| + t ) \; dt \\
&\le (\bb{E} \| \tilde{A}_0 \| )^2 + 2 \int_0^\infty (\bb{E} \| \tilde{A}_0 \|  + t) e^{- t^2/4} \; dt \\
& = (\bb{E} \| \tilde{A}_0 \| )^2 + 2\sqrt{\pi}\, \bb{E} \| \tilde{A}_0 \| + 4 \\
&\le (\bb{E} \| \tilde{A}_0 \| + 2)^2.
\end{align*}
Setting $t = 2\sqrt{2 \log n}$ in (\ref{ineqn::tailA0-2}), we deduce $\| \tilde{A}_0 \| \ge \bb{E} \| \tilde{A}_0 \|-  2\sqrt{2\log n} \ge (\bb{E} \| \tilde{A}_0 \|^2)^{1/2} - 2\sqrt{2\log n} - 2$ with probability $1 - O(n^{-2})$. It is sufficient, therefore, to prove
\begin{equation*}
\bb{E} \| \tilde{A}_0 \|^2 > (\lambda_1 + 2\sqrt{2 \log n} + 2 )^2.
\end{equation*}
Observe that 
\begin{equation*}
\bb{E} \| \tilde{A}_0 \|^2 = \bb{E} \max_{\| u \|_2 = 1} \| A_0u + Gu \|_2^2 \ge \max_{\| u \|_2 = 1} \bb{E} \| A_0u + Gu \|_2^2.
\end{equation*}
The $j$'th entry of $Gu$ is a Gaussian variable with variance $1 + u_j^2$, so
\begin{equation}\label{ineqn::lambda1-2}
\bb{E} \| \tilde{A}_0 \|^2 \ge \max_{\| u \|_2 = 1} \bb{E} \| A_0u \|_2^2 + n = \lambda_2^2 +n,
\end{equation}
where the maximum is attained at $u = (1,0,\ldots,0)^T$. Thus, it boils down to showing 
\begin{equation*}
\lambda_2^2 +n > (\lambda_1 + 2\sqrt{2 \log n} + 2 )^2.
\end{equation*}
To verify it, we use $\lambda_1 - \lambda_2 = n^{1/p}/\log^2 n$ and $\lambda_1 - \lambda_n = (n-1)^{(p-2)/p} \cdot n^{1/p}/\log^2 n$ to obtain
\begin{align*}
~~~~ &(\lambda_1 + 2\sqrt{2 \log n} + 2 )^2 - \lambda_2^2 \\
&\le (\lambda_1 - \lambda_2 + 2\sqrt{2\log n} + 2) (2\lambda_1 + 2\sqrt{2\log n} + 2) \\
&\le \Big(\frac{n^{1/p}}{\log^2 n} + O(\sqrt{\log n})\Big) \Big( \frac{2(n-1)^{(p-2)/p} \cdot n^{1/p}}{\log^2 n} + 2\lambda_n + O(\sqrt{\log n}) \Big)\\
&\le \Big(\frac{n^{1/p}}{\log^2 n} + O(\sqrt{\log n})\Big) \Big( \frac{2n^{1 - 1/p}}{\log^2 n} + O(\sqrt{ n}) \Big)\\
&= o(n).
\end{align*}
The last line is due to $p \ge 2$. Thus, for large enough $n$, (\ref{ineqn::lambda1-2}) is true, which completes the proof.

\end{proof}

\subsection{Additional Proofs in Section \ref{sec::pre}}
\begin{proof}[Proof of Theorem \ref{thm::Chevet}]
Let us write $X_u = \langle u, Xu \rangle$. For each $u \in T$, $X_u$ is a linear combination of $(\Re(X_{jk}))_{j \le k}$ and $(\Im(X_{jk}))_{j \le k}$, so $X_u$ is also sub-gaussian.  For any $u,v \in T$, 
\begin{align*}
X_u - X_v &= \sum_{j=1}^m X_{jj} ( | u_j |^2 - |v_j|^2 ) + 2\sum_{j > k} \Re(X_{jk}) \Re( \bar{u}_j u_k - \bar{v}_j v_k ) \\
& - 2\sum_{j > k} \Im(X_{jk}) \Im(\bar{u}_j  u_k - \bar{v}_j v_k).
\end{align*}
So, by independence, we can bound the $\psi_2$ norm of $X_u - X_v$ (Lemma 5.9 in \cite{Ver10}).
\begin{align*}
\| X_u - X_v \|_{\psi_2}^2 &\le C^2 \sum_{j=1}^m ( | u_j |^2 - |v_j|^2 )^2 + 4C^2 \sum_{j > k} | \bar{u}_j u_k - \bar{v}_j v_k |^2 \\
&\le 2C^2 \| uu^* - vv^* \|_F^2 \\
&\le 8C^2 \rad(T)^2 \| u - v \|_2^2,
\end{align*}
where $C>0$ is an absolute constant. We used the fact that $\| u \|_2 \le \rad(T), \| v \|_2 \le \rad(T)$ in the last inequality. Let $G_u := 4C \Re \langle u, g \rangle$ be a Gaussian process, where $g$ is a standard complex Gaussian vector. It is straightforward to check that $\| G_u - G_v \|_2^2 = 8C^2\|u - v\|_2^2$, so (\ref{ineqn::increment}) holds with $M=\rad(T)$. From Fernique-Talagrand's comparison theorem, we conclude 
\begin{equation*}
\bb{E} \sup_{u \in T} X_u \lesssim \rad(T) \, \bb{E} \sup_{u \in T} \Re \langle u, g \rangle,
\end{equation*}
which results in the desired inequality.
\end{proof}

~

\begin{proof}[Proof of Theorem \ref{thm::opNorm}]
In this proof, we use $C_1, C_2,\ldots,C_5 > 0$ and $c>0$ to denote different absolute constants.

(i) Notice that for $p \in [2,\infty)$,
\begin{equation*}
\| X \|_{p',p} = \sup_{\| u \|_{p'} \le 1} \| X u \|_{p} = \sup_{\|u\|_{p'} \le 1, \|v\|_{p'} \le 1} | \langle v, Xu \rangle | \le  4\sup_{ \| u \|_{p'} \le 1} | \langle u, Xu \rangle |,
\end{equation*}
where we used duality of $L^p$ spaces and polarization identity for Hermitian matrix $X$:
\begin{equation*}
4\langle v, Xu \rangle = X_{u+v} - X_{u-v} +i X_{u +iv} - iX_{u-iv},
\end{equation*}
where $X_u := \langle u, Xu \rangle$. Since $X_u$ takes real values, we can apply Chevet's inequality (Theorem \ref{thm::Chevet}) for both $X$ and $-X$, and obtain 
\begin{align*}
\bb{E} \| X \|_{p',p}  &\le 4 \bb{E} \sup_{\eta \in \{ \pm 1\}} \sup_{ \| u \|_{p'} \le 1} \langle u, \eta Xu \rangle \le 4\bb{E}  \sup_{ \| u \|_{p'} \le 1} \langle u,  Xu \rangle  + 4\bb{E}  \sup_{ \| u \|_{p'} \le 1} \langle u, -Xu \rangle  \\
& \lesssim w(T) \rad(T),
\end{align*} 
where $T= \{ u \in \bb{C}^m:  \| u \|_{p'} \le 1 \}$ is the unit $\ell^{p'}$ ball. Since $1 < p' \le 2$, the radius $\rad(T)$ is bounded by $1$, due to the trivial inequality $\| u \|_{2} \le \| u \|_{p'}$. To bound $w(T)$, we use Jensen's inequality to derive
\begin{equation*}
w(T) = \bb{E} \sup_{\| u \|_{p'} \le 1} | \langle u, g \rangle | = \bb{E} \| g \|_{p} \le ( \bb{E} \| g \|_p^p )^{1/p} = C_p^{1/p} n^{1/p},
\end{equation*}
where $C_p$ is the $p$'th moment of a complex standard Gaussian variable. A Gaussian variable is of course subgaussian, so $C_p^{1/p} \lesssim \sqrt{p}$. Thus, we have proved $\bb{E} \| X \|_{p',p} \lesssim \sqrt{p}\, n^{1/p}$, which concludes the first part of this theorem.\\

(ii) To study the concentration of $\ell^{p'} \to \ell^p$ norm, first we introduce a truncated random matrix $\widetilde{X}$, defined by 
\begin{equation*}
\tilde{X}_{jk} = \Re(X_{jk}) \mathbf{1}\{ |\Re(X_{jk}) | \le \tau \sqrt{\log n} \} + i \Im(X_{jk}) \mathbf{1} \{ |\Im(X_{jk})| \le \tau \sqrt{\log n} \}, 
\end{equation*}
for any $1 \le j \le k \le n$, where $\tau$ is to be determined later. For $j > k$ define $\tilde{X}_{jk} = \overline{ \tilde{X}_{kj}}$. This definition assures that $\widetilde{X}$ is Hermitian, that $\Re(\tilde{X}_{jk}), \Im(\tilde{X}_{jk})$ are independent for $1 \le j \le k \le n$, and that both real and imaginary parts are uniformly bounded by $\tau \sqrt{\log n}$. Though $\widetilde{X}_{jk}$ is no longer centered, its mean is negligible when $\tau$ is large. To justify this, for any centered random variable $\xi$ with $\| \xi \|_{\psi_2} \le 1$, we have tail probability bound $\bb{P} ( | \xi | > t ) \le \exp( 1 - t^2 / C_1^2)$ for all $t>0$ \citep{Ver10}, and this leads to
\begin{align}
~~~~ & ~ 
\bb{E} \, | \xi | \mathbf{1} \{ | \xi | > \tau \sqrt{\log n} \}  \notag \\
= & ~ \tau \sqrt{\log n} \, \bb{P} ( | \xi | > \tau \sqrt{\log n} ) + \int_{0}^{\infty} \bb{P} ( | \xi | > \tau \sqrt{\log n} + t) \; dt \notag \\
\le& ~ \tau \sqrt{\log n}\, e^{1 - \tau^2 \log n / C_1^2} + e^{1 - \tau^2 \log n / C_1^2} \int_0^\infty e^{-t^2/C_1^2} \; dt \notag \\
\le & ~ (e \tau \sqrt{\log n}  + eC_1 \sqrt{\pi}/2) n^{-\tau^2/C_1^2} \notag \\
\lesssim & ~ n^{-2}, \label{ineqn::truncDev}
\end{align}
where we chose $\tau = 3C_1$ in the last inequality. To invoke Talagrand's concentration inequality, we observe that for any Hermitian matrix $M \in \bb{C}^{n \times n}$,
\begin{equation*}
\| M \|_{p',p} \le 4 \sup_{\| u \|_{p'} \le 1 } | \langle u, Mu \rangle | \le 4 \sup_{\| u \|_2 \le 1 } | \langle u, Mu \rangle | \le 4 \| M \|_F.
\end{equation*}
Thus for any Hermitian $M_1,M_2 \in \bb{C}^{n \times n}$, 
\begin{align}
| \| M_1 \|_{p',p} - \| M_2 \|_{p',p} | &\le \| M_1 - M_2 \|_{p',p} \le 4\|M_1 - M_2 \|_F \notag \\
&\le 8 \big( \sum_{j \le k} |(M_1 - M_2)_{jk} |^2 \big)^{1/2}, \label{ineqn::p'pLip}
\end{align}
so $\| \cdot \|_{p',p}$ is $8$-Lipschitz as a function of $\{ \Re(M_{jk}), \Im(M_{jk}) \}_{j \le k}$. Moreover, $\| \cdot \|_{p',p}$ is a convex function since it is a norm. Applying Talagrand's concentration inequality \citep{Tal95} to $\| \tilde{X} \|_{p',p}$, we deduce 
\begin{equation*}
\bb{P}( | \| \tilde{X} \|_{p',p} - \bb{E} \|  \tilde{X} \|_{p',p} | \ge t\tau \sqrt{\log n} ) \le C_2 \exp \left( -c t^2 \right),
\end{equation*}
Setting $t = C_3\sqrt{\log n}$ where $C_3 > 0$ is some large constant, we have $ | \| \tilde{X} \|_{p',p} - \bb{E} \|  \tilde{X} \|_{p',p} | \le 2C_1C_3 \log n$ with probability $1 - O(n^{-4})$. To relate $\tilde{X}$ to $X$, we use (\ref{ineqn::p'pLip}) to derive 
\begin{equation*}
\big|  \|  \tilde{X} \|_{p',p} -  \|  X\|_{p',p} \big| \le 8 \big( \sum_{j \le k} |(\tilde{X} - X)_{jk} |^2 \big)^{1/2} \le 8  \sum_{j \le k}|(\tilde{X} - X)_{jk} |,
\end{equation*}
and therefore, using the bound (\ref{ineqn::truncDev}) for both the real and imaginary parts, we have
\begin{equation*}
\big| \bb{E} \|  \tilde{X} \|_{p',p} - \bb{E} \| X \|_{p',p} \big| \le 8 \sum_{j \le k} \bb{E} |(\tilde{X} - X)_{jk} | \le 16\sum_{j \le k}  C_4 n^{-2} \le  16C_4,
\end{equation*}
where $C_4 > 0$ is the constant hiding in (\ref{ineqn::truncDev}). Finally, we can bound $\| X \|_{p',p} - \bb{E} \| X \|_{p',p}$ using (i) and the concentration of $\tilde{X}$. Let us write $A_{jk} = \{ \Re(\tilde{X}_{jk}) \neq \Re(X_{jk}) \text{ or } \Im(\tilde{X}_{jk}) \neq \Im(X_{jk}) \}$ for shorthand. From the tail probability bound of sub-gaussian variables and the choice of $\tau$, we have $\max_{j\le k} \bb{P}(A_{jk}) \lesssim n^{-6}$. Thus, for a large absolute constant $C_5$ such that $C_5 \log n - 16C_4 \ge 2C_1C_3 \log n$, we derive
\begin{align*}
~~~~ & ~  \bb{P}(  | \| X \|_{p',p} - \bb{E} \| X \|_{p',p} | \ge C_5 \log n ) \\
\le& ~  \bb{P}(  | \| \tilde{X} \|_{p',p} - \bb{E} \| X \|_{p',p} | \ge C_5 \log n ) + \bb{P}( \bigcup_{j \le k} A_{jk} ) \\
\le & ~ \bb{P}( | \| \tilde{X} \|_{p',p} - \bb{E} \| \tilde{X} \|_{p',p} | \ge C_5 \log n - 16C_4 ) + \sum_{j \le k} \bb{P}(A_{jk}) \\
\lesssim & ~ n^{-4} + \sum_{j \le k} n^{-6} \lesssim n^{-4}.
\end{align*} 
This finishes the proof of part (ii). To conclude the proof, we combine the two bounds in (i) and (ii) and derive
\begin{equation*}
\| X \|_{p',p} \lesssim \log n + \sqrt{p}\, n^{1/p} \lesssim \sqrt{p \log n}\, n^{1/p},
\end{equation*}
where the second $\lesssim$ is due to $\sqrt{\log n} \lesssim \min_{p \in [2,\infty]}  \sqrt{p}\, n^{1/p}$.
\end{proof}

~

\begin{proof}[Proof of Corollary \ref{cor::conctrNorm}]
Without loss of generality we can assume that $\xi$ is a real vector (incurring a constant of at most $2$). Let us consider a special random matrix $X$ in Theorem \ref{thm::opNorm}. Let $X_{11} = 0, X_{j+1,1} = \xi_j$ for $j=1,\ldots,n$, and $X_{jk} = 0 $ for $2 \le j,k \le n+1$. We only need to show $\| \xi \|_p \le \| X \|_{p',p}$. This is true because
\begin{equation*}
\| X \|_{p',p} = \sup_{\substack{\| u \|_{p'} \le 1 \\ u \in \bb{C}^{n+1}}} \| X u \|_p \ge \sup_{\substack{\| u \|_{p'} \le 1 \\ u \in \bb{C}^{n+1}}} | (X u)_1 |  \ge \sup_{\substack{\| u \|_{p'} \le 1 \\ u \in \bb{C}^{n}}} | \langle \xi, u \rangle | = \| \xi \|_p,
\end{equation*}
where we used duality of $L_p$ spaces.
\end{proof}

~

\begin{proof}[Proof of Theorem \ref{thm::eigvalPtb}]

Without loss of generality we can assume $\mu_1 \ge \ldots \ge \mu_n > 0$. Let us denote $\rho(z) := z^* D_\mu z = \sum_{j=1}^n \mu_j |z_j|^2$ for $z \in \bb{C}^n$; we also denote $X_z := z^* X z + \tau \| z \|_2 \Re(g^* z)$. First observe that the desired inequality (\ref{ineqn::eigval1}) is equivalent to
\begin{align*}
& \qquad ~~ \sup_{z \in \bb{C}^n} (X_z - \rho(z)) \le 0, \\
&\Leftrightarrow ~ \sup_{z: \rho(z) = \mu_n} ( X_z - \rho(z) ) \le 0, \\
&\Leftrightarrow ~ \sup_{z: \rho(z) \le \mu_n}  X_z \le \mu_n,
\end{align*}
where we used homogeneity of $\rho(z)$ and $X_z$, meaning that $\rho(az) = a^2 \rho(z)$ and $X_{az} = a^2 X_z$ for any $a \in \bb{R}^+$; and the second equivalence is due to the fact that maxima must be attained  either when $z=0$ or on the boundary of the set $\{z:\rho(z) \le \mu_n\}$. We can identify $\bb{C}^n$ with $\bb{R}^{2n}$ in the obvious way (and view $z$ also as an $2n$-dimensional real vector). Then the index set $\{z:\rho(z) \le \mu_n\}$ becomes
\begin{equation*}
\Omega := \big\{ x \in \bb{R}^{2n}: \sum_{j=1}^n \frac{ |x_{2j-1}|^2 +  |x_{2j}|^2}{ \mu_n / \mu_j} \le 1 \big\}.
\end{equation*}
Note that $\Omega \subset B_1^{2n}$, the unit ball in $\bb{R}^{2n}$. To investigate $\sup_{z\in \Omega}  X_z$, we observe that $X_z$ is a subgaussian process, which allows us to employ powerful tools developed in supreme of processes. For any complex $u,v$ with $\| u \|_2 \le 1, \| v \|_2 \le 1$, $X_u - X_v$ is a linear combination of independent subgaussian variables, so $X_u - X_v$ is subgaussian, and 
\begin{align*}
\| X_u - X_v \|_{\psi_2}^2 &\le C^2 \sum_{j=1}^n ( | u_j |^2 - |v_j|^2 )^2 + 4C^2 \sum_{j > k} | \bar{u}_j u_k - \bar{v}_j v_k |^2 \\
&~+ C^2\tau^2 \sum_{j=1}^n \big| \| u \|_2 u_j - \| v \|_2 v_j \big|^2 \\
&\le 2C^2 \| uu^* - vv^* \|_F^2 +2C^2 \| u - v \|_2^2 ( \| u \|_2^2 + \|v \|_2^2)\\
&\le 12C^2 \| u - v\|_2^2.
\end{align*}

The classical chaining argument (\cite{pisier1983some}; also see \cite{Van14} Theorem 5.29) shows 
\begin{equation}\label{ineqn::chain}
\bb{P}( \sup_{z \in \Omega}  X_z \ge C_0 \int_0^{\infty} \sqrt{\log N(\Omega, \| \cdot \|_2, \varepsilon)} \; d\varepsilon + t ) \le C_0e^{-t^2/C_0},
\end{equation} 
Note $C, C_0>0$ are absolute constants. Apparently, since $\rad(\Omega) \le 1$, when $\varepsilon > 1$ the covering number $N(\Omega, \| \cdot \|_2, \varepsilon) = 1$. For $\varepsilon \le 1$, we let the ellipsoid $E_a = \Omega/\varepsilon$; in other words, the axes of the ellipsoid are $a_{2j-1} = a_{2j} = \varepsilon^{-1} \sqrt{\mu_n/\mu_j} $ for $j=1, \ldots, n$. Applying Lemma \ref{lem::covering} to $E_a$ and setting $\theta = 1/4$, we obtain
\begin{equation*}
\log N(\Omega, \| \cdot \|_2, \varepsilon) \le 2\log h(a) \cdot   | \{ j: \mu_n / \mu_j > \varepsilon^2 \}|  + 2\log(4\bar{C}) \cdot | \{ j: \mu_n / \mu_j > \varepsilon^2/2 \}|
\end{equation*}
for any $\varepsilon > 0$. Using a trivial inequality $\sqrt{x + y} \le \sqrt{x} + \sqrt{y}$, we only need to bound two integrals
\begin{equation}\label{eqn::Int}
I_1:= \int_0^{1} \sqrt{ | \{ j: \mu_n / \mu_j > \varepsilon^2/2 \}| } \; d\varepsilon,  \quad 
I_2:= \int_0^1  \sqrt{ 
\log h(a) \cdot  | \{ j: \mu_n / \mu_j > \varepsilon^2 \}|}\; d\varepsilon .
\end{equation}
The first integral $I_1$ is bounded by $\int_0^{\sqrt{2}} \sqrt{ | \{ j: \sqrt{2\mu_n / \mu_j} > \varepsilon \}| } \; d\varepsilon$, whose integrand is a step function with jumps at $\sqrt{2\mu_n / \mu_j}, j=1,\ldots,n$. So an upper bound of $I_1$ is
\begin{align*}
I_1 &\le  \sum_{k=1}^n \sqrt{n-k+1} ( \sqrt{2 \mu_n / \mu_k} - \sqrt{2 \mu_n / \mu_{k-1}}) \qquad \mu_0 :=\infty \\
&= \sum_{k=1}^n \sqrt{2 \mu_n/ \mu_k} (\sqrt{n-k+1} - \sqrt{n-k})\\
&\le \sqrt{2} + \sqrt{\mu_n/2} \sum_{k=1}^{n-1} ( (n-k) \mu_k)^{-1/2}.
\end{align*}
where we used $\sqrt{n-k+1} - \sqrt{n-k} \le (2\sqrt{n-k})^{-1}$ in the last inequality.

To bound the second integral $I_2$, we use $\log h(a) \le \log a_{2n} = \log (\varepsilon^{-1})$. Thus, 
\begin{equation*}
\log h(a) \cdot  | \{ j: \mu_n / \mu_j > \varepsilon^2 \}| \le \begin{cases}
n \log (\varepsilon^{-1}), & 0 \le \varepsilon \le 1/n,  \\
\log n \cdot | \{ j: \sqrt{\mu_n / \mu_j} > \varepsilon \}|, & 1/n < \varepsilon \le 1.
\end{cases}
\end{equation*}
So we can split $I_2$ into integrals over subinterval $[0,1/n]$ and $[1/n, 1]$, and this yields
\begin{align*}
I_2 &\le \sqrt{n} \int_0^{1/n} \sqrt{ \log (\varepsilon^{-1}) }\; d\varepsilon + \sqrt{\log n} \int_{1/n}^1 | \{ j: \sqrt{\mu_n / \mu_j} > \varepsilon \}| \; d\varepsilon \\
&\le \sqrt{n} \int_0^{1/n} \sqrt{ \log (\varepsilon^{-1}) }\; d\varepsilon+ \sqrt{\log n} \int_0^1 | \{ j: \sqrt{\mu_n / \mu_j} > \varepsilon \}| \; d\varepsilon \\
&\le \sqrt{n} \int_0^{1/n} \sqrt{ \log (\varepsilon^{-1}) }\; d\varepsilon + \sqrt{\log n  \cdot \mu_n} \sum_{k=1}^{n-1} ( (n-k) \mu_k)^{-1/2} /2 + \sqrt{\log n}.
\end{align*}
The last inequality is derived similarly as we have shown for $I_1$. By change of variables,
\begin{equation*}
\sqrt{n} \int_0^{1/n} \sqrt{ \log \varepsilon^{-1} }\; d\varepsilon = n^{-1/2} \int_0^1 \sqrt{ \log (n/z)} \; dz \le n^{-1/2} \int_0^1 \sqrt{\log n} + \sqrt{\log(z^{-1})} \;dz.
\end{equation*}
Since $\int_0^1 \sqrt{\log (z^{-1})} \; dz < \infty$, we have $\sqrt{n} \int_0^{1/n} \sqrt{ \log \varepsilon^{-1} }\; d\varepsilon  \le \sqrt{\log n/n} + C_1 n^{-1/2}$, where $C_1>0$ is an absolute constant. Now with these upper bounds on $I_1$ and $I_2$, we are able to bound $\int_0^{\infty}\sqrt{\log N(\Omega, \| \cdot \|_2, \varepsilon)} \; d\varepsilon$.
\begin{equation*}
\int_0^{\infty}\sqrt{\log N(\Omega, \| \cdot \|_2, \varepsilon)} \; d\varepsilon \le C_2 \sqrt{\log n \cdot \mu_n} \sum_{k=1}^{n-1} ( (n-k) \mu_k)^{-1/2} + C_2\sqrt{\log n} ,
\end{equation*}
where $C_2 >0$ is an absolute constant. For any $p \in [2, \infty)$, we can use H\"{o}lder's inequality to derive
\begin{align*}
\sum_{k=1}^{n-1} ( (n-k) \mu_k)^{-1/2} &\le \big( \sum_{k=1}^{n-1} (n-k)^{-p/(p+2)} \big)^{(p+2)/2p} \cdot \big( \sum_{k=1}^{n-1} \mu_k^{-p/(p-2)} \big)^{(p-2)/2p} \\
&\le \big( \int_0^{n-1} x^{-p/(p+2)} \; dx\big)^{(p+2)/2p} \cdot \| \ddiag(D_\mu^{-1}) \|_{p/(p-2)}^{1/2} \\
&\le C_p n^{1/p} \| \ddiag(D_\mu^{-1}) \|_{p/(p-2)}^{1/2},
\end{align*}
where $C_p = ( (p+2)/2 )^{(p+2)/2p}$. Note that the bound is valid for $p=2$, in which case $\| \cdot \|_{p / (p-2)}$ is simply the $\ell^\infty$ norm. Since $(p+2)/2 \asymp p$ and $p^{1/p} \asymp 1$, we have $C_p \asymp \sqrt{p}$. Suppose the absolute constant $c_0>0$ is small enough. Then Assumption \ref{assmp::mu} implies the following inequalities
\begin{equation}\label{ineqn::eigvalCond}
\begin{cases}
C_p \sqrt{\log n}\, n^{1/p} \| \ddiag(D_\mu^{-1}) \|_{p/(p-2)} \le 1/4C_0C_2, \\
\mu_n^{-1} \sqrt{\log n} \le 1/4C_0C_2, \\
\mu_n^{-1} \sqrt{\log n} \le 1 / 4 \sqrt{C_0}.
\end{cases}
\end{equation}
The first inequality in (\ref{ineqn::eigvalCond}) implies 
\begin{equation}
C_p \sqrt{\log n}\, n^{1/p} \mu_n^{-1} \le C_p \sqrt{\log n}\, n^{1/p} \| \ddiag(D_\mu^{-1}) \|_{p/(p-2)} \le 1/4C_0C_2.
\end{equation}
It follows that $C_p^2 \log n \cdot n^{2/p} \| \ddiag(D_\mu^{-1}) \|_{p/(p-2)} \le \mu_n/(4C_0C_2)^2$. Taking square roots on both sides yields
\begin{equation*}
C_2C_p \sqrt{\log n \cdot \mu_n}\, n^{1/p} \| \ddiag(D_\mu^{-1}) \|_{p/(p-2)}^{1/2} \le \mu_n/4C_0.
\end{equation*}
The second inequality in (\ref{ineqn::eigvalCond}) implies  $C_2\sqrt{\log n} \le \mu_n/4C_0$. Thus we deduce
\begin{equation}\label{ineqn::coveringItgbound}
\int_0^{\infty}\sqrt{\log N(\Omega, \| \cdot \|_2, \varepsilon)} \; d\varepsilon \le \mu_n/2C_0.
\end{equation}
If $p = \infty$, H\"{o}lder's inequality gives
\begin{equation*}
\sum_{k=1}^{n-1}( (n-k)\mu_k)^{-1/2} \le  \big( \sum_{k=1}^{n-1} (n-k)^{-1} \big)^{1/2} \cdot \big( \sum_{k=1}^{n-1} \mu_k^{-1} \big)^{1/2} \le \sqrt{1 + \log n} \cdot \| \ddiag(D_{\mu}^{-1}) \|_{1}^{1/2}.
\end{equation*}
For suitably small $c_0 > 0$, the following holds
\begin{equation}\label{ineqn::eigvalCond2}
\begin{cases}
\sqrt{\log n ( 1 + \log n)}\,  \| \ddiag(D_\mu^{-1}) \|_{1} \le 1/4C_0C_2, \\
\mu_n^{-1} \sqrt{\log n} \le 1/4C_0C_2, \\
\mu_n^{-1} \sqrt{\log n} \le 1 / 4 \sqrt{C_0}.
\end{cases}
\end{equation}
As before, the first two inequalities above imply (\ref{ineqn::coveringItgbound}).

To conclude with a `high probability' result, we set $t = \mu_n/2$ in (\ref{ineqn::chain}). For $p \in [2,\infty]$, from (\ref{ineqn::eigvalCond}) and (\ref{ineqn::eigvalCond2}) we know $\mu_n \ge 4 \sqrt{C_0\log n}$, we conclude that with probability $1 - O(n^{-4})$, $\sup_{z \in \Omega} X_z \le \mu_n$, which finishes the proof.

\end{proof}

\subsection{Additional Proofs in Section \ref{sec::L-1}}

\begin{proof}[Proof of Lemma \ref{lem::L-1norm}]
Since $D$ is invertible, $\cal{L}$ is invertible if and only if $I_{n-1} - E_{22} D^{-1}$ is invertible. For any $y \in \bb{C}^{n-1}$, it is easy to see that $u \mapsto y + E_{22} D^{-1} u$ is a contraction mapping under $\ell^p$ norm, so by contraction mapping theorem, there is a unique fixed-point, which is a solution to equation $(I_{n-1} - E_{22} D^{-1})x = y$. This proves invertibility of $\cal{L}$. 
From the identity $(I_{n-1} - E_{22} D^{-1} ) D \cal{L}^{-1} u = u$, we have
\begin{equation*}
\| u \|_p \ge \| D \cal{L}^{-1} u \|_p - \| E_{22}D^{-1} \cdot D\cal{L}^{-1} u \|_p \ge \| D \cal{L}^{-1} u \|_p / 2.
\end{equation*}
Thus, $\| D \cal{L}^{-1} u \|_p \le 2\|u\|_p$.
\end{proof}

The next two easy lemmas are useful for norm conversion.

\begin{lem}\label{lem::diagMatNorm}
Suppose $p \in [2,\infty]$, and $z \in \bb{C}^n$ is a vector. There is an identity between matrix operator norm of the diagonal matrix $\diag(z)$ and vector norm of $z$: $\| \diag(z) \|_{p,p'} = \| z \|_{p/(p-2)}$.
\end{lem}
\begin{proof}
This identity is trivial for $p=\infty$ and $p=2$. For $p \in (2,\infty)$, we let $q = p/p'$, where $p'$ satisfies $1/p + 1/p' = 1$. Since $q \in (1,\infty)$, we can also let $q' \in (1,\infty)$ be such that $1/q + 1/q' = 1$. For any $u \in \bb{C}^n$ with $\| u \|_p \le 1$, by H\"{o}lder's inequality,
\begin{equation*}
\| \diag(z) u \|_{p'}^{p'} = \sum_{j=1}^n |z_j u_j |^{p'} \le \big( \sum_{j=1}^n |z_j|^{p'q'} \big)^{1/q'} \big( \sum_{j=1}^n |u_j|^{p'q} \big)^{1/q}.
\end{equation*}
Since $p'q = p$ and $\| u \|_p \le 1$, it follows that $\| \diag(z) u \|_{p'} \le \| z \|_{p'q'} = \| z \|_{p/(p-2)}$. The equality is attained for $u$ with $\| u \|_p = 1$ and $|u_j| \propto |z_j|^{1/(p-2)}$. This proves the identity $\| \diag(z) \|_{p,p'} = \| z \|_{p/(p-2)}$.
\end{proof}

\begin{lem}\label{lem::simpleNormBound}
Suppose $x \in \bb{R}^N$, and $p \in [2,\infty)$. Then, 
\begin{equation*}
\| x \|_2 \le N^{1/p} \| x \|_{p/(p-2)}, \qquad \forall p \in [2,\infty).
\end{equation*}
\end{lem}
\begin{proof}
For any vector $y \in \bb{R}^N$ and $1 \le s \le r \le \infty$, there is a trivial bound between vector norms $\| y \|_r \le \| y \|_s \le N^{1/s - 1/r} \| x \|_r$. If $p \ge 4$, the result is trivial; if $p \in [2,4)$, since $1/2 - (p-2)/p = (4-p)/2p \le 1/p$, the result also follows.
\end{proof}

Now we can prove Lemma \ref{lem::indct}. The proof relies on Theorem \ref{thm::opNorm} and basic properties of subgaussian variables.
\begin{proof}[Proof of Lemma \ref{lem::indct}]
In this proof, $C_0,C_1,\ldots, C_6$ are also some positive absolute constants.\\
(i) From an equivalent definition of subgaussian random variables (see \citep{Ver10}), each entry of $\tilde{E}$ can be bounded by $C_0\sqrt{\log n}$ with probability $1- O(n^{-6})$, where $C_0>0$ is an absolute constant. Thus it follows from union bound that with probability $1 - O(n^{-4})$, $\| \tilde{E} \|_{\infty} \le C_0 \sqrt{\log n}$ .\\

(ii) Note that $D = D_\lambda + E_{11}I_{n-1}$. With suitably small constant $c_0$, Assumption \ref{assmp::lambda} implies $\delta \ge 2C_0  \sqrt{\log n}$. As a consequence of part (i), with probability $1 - O(n^{-4})$, each diagonal entry of $D$ satisfies $D_{jj} \ge (D_\lambda)_{jj}/2$, so $\| D^{-1} \|_{p,p'} \le 2 \| D_\lambda^{-1} \|_{p,p'}$. To prove (\ref{ineqn::lemIndct1}), let us first consider $2 \le p < \infty$. From Theorem \ref{thm::opNorm}, we know that with probability $1 - O(n^{-4})$, $\| E_{22} \|_{p',p} \le C_1\sqrt{p \log n}\, n^{1/p}$. Thus,
\begin{equation*}
\| E_{22} D^{-1} \|_{p,p} \le \| E_{22} \|_{p',p} \| D^{-1} \|_{p,p'} \le 2C_1\sqrt{p \log n}\, n^{1/p}\| D_\lambda^{-1} \|_{p,p'} . 
\end{equation*}
By Lemma \ref{lem::diagMatNorm}, $\| D_\lambda^{-1} \|_{p,p'} = \| \ddiag(D_\lambda^{-1}) \|_{p/(p-2)}$. Thus, under Assumption \ref{assmp::lambda}, $\| E_{22} D^{-1} \|_{p,p} \le 2c_0C_0C_1$, which is smaller than $1/2$ when $c_0$ is small enough. This proves the first inequality in (\ref{ineqn::lemIndct1}), and the second inequality follows similarly.
 
By Corollary \ref{cor::conctrNorm}, we have $\| E_{21} \|_p \le C_2 \sqrt{p \log n}\, n^{1/p}$. In the same way as just derived, we have
\begin{equation*}
\| D^{-1}E_{21} \|_{p'} \le \| D^{-1} \|_{p,p'} \| E_{21} \|_{p} \le 2C_2\sqrt{p \log n}\, n^{1/p} \| \ddiag(D_\lambda^{-1}) \|_{p/(p-2)} \le 1/2
\end{equation*}
for a small constant $c_0$. 

For the case $p = \infty$, we deduce from part (i) that, with probability $1 - O(n^{-4})$,
\begin{equation} \label{ineqn::E22InfNorm}
\max\{\| E_{22} D^{-1} \|_{\infty, \infty},  \| E_{22}^{(m)} D^{-1} \|_{\infty, \infty}, \| D^{-1} E_{21} \|_1\} \le C_0 \sqrt{\log n} \| \ddiag(D^{-1}) \|_{1} \le 1/2,
\end{equation}
The second inequality follows from $ \| \ddiag(D^{-1}) \|_{1} \le 2 \| \ddiag(D_\lambda^{-1}) \|_{1}$ and Assumption \ref{assmp::lambda} with constant $c_0$ being small enough.\\

(iii) First consider $2 \le p < \infty$. Without loss of generality, we can assume that $\xi$ is a deterministic vector, since we can always condition on $\xi$ in the first place.  The $m$'th coordinate of $\Delta E_{22}^{(m)} D^{-1} \xi$ is $\sum_{j=1}^{n-1}(E_{22})_{mj} \xi_j / D_{jj}$, which is a sum of independent subgaussian random variables. So we can bound its $\psi_2$ norm (see \cite{Ver10}) as follows.
\begin{equation*}
\big\| \sum_{j=1}^{n-1}(E_{22})_{mj} \xi_j / D_{jj} \big\|_{\psi_2}^2 \le C_3 \sum_{j=1}^{n-1}  D_{jj}^{-2} | \xi_j|^2 \le C_3 \sum_{j=1}^{n-1}  D_{jj}^{-2} \| \xi \|_{\infty}^2.
\end{equation*}
It follows that with probability $1-O(n^{-4})$, 
\begin{equation*}
\big| \sum_{j=1}^{n-1}(E_{22})_{mj} \xi_j / D_{jj} \big| \le C_4 \sqrt{\log n}\, \big( \sum_{j=1}^{n-1} D_{jj}^{-2} \big)^{1/2} \| \xi \|_{\infty}.
\end{equation*}
For $j \neq m$, the $j$'th coordinate of $\Delta E_{22}^{(m)} D^{-1} \xi$ is $(E_{22})_{jm} \xi_m / D_{mm}$. From Corollary \ref{cor::conctrNorm}, we can bound the norm of $m$'th column of $\Delta E_{22}^{(m)}$: $\| (\Delta E_{22}^{(m)})_{\cdot m} \|_p \le C_5 \sqrt{p \log n}\, n^{1/p}$, and this leads to
\begin{align*}
\| \Delta E_{22}^{(m)} D^{-1} \xi \|_p &\le | (\Delta E_{22}^{(m)} D^{-1} \xi)_m | + \Big( \sum_{j \ne m} |(\Delta E_{22}^{(m)} D^{-1} \xi)_j|^p \Big)^{1/p} \\
&\le C_4 \sqrt{\log n}\, \big( \sum D_{jj}^{-2} \big)^{1/2} \| \xi \|_{\infty} + C_5\sqrt{p \log n}\, n^{1/p} D_{mm}^{-1} | \xi_m| \\
&\le 2C_4  \sqrt{\log n}\, \| \ddiag(D_\lambda^{-1}) \|_2 \| \xi \|_{\infty} + 2C_5 \sqrt{p \log n}\, n^{1/p} \delta^{-1} \| \xi \|_{\infty} 
\end{align*}
with probability $1-O(n^{-4})$, where we used $D_{jj} \ge \lambda_1 - \lambda_j - C_0\sqrt{\log n} \ge \delta/2$ due to part (i) and Assumption \ref{assmp::lambda}. Lemma \ref{lem::simpleNormBound} gives a useful conversion between vector norms, and we have 
\begin{equation}\label{DlambdaConv}
\| \ddiag(D_\lambda^{-1}) \|_2 \le n^{1/p} \| \ddiag(D_\lambda^{-1}) \|_{p/(p-2)}, \qquad \forall p \in [2,\infty).
\end{equation}
We conclude that for some small constant $c_0$, $\| \Delta E_{22}^{(m)} D^{-1} \xi \|_p \le \| \xi \|_{\infty}/12$ holds with probability $1-O(n^{-4})$. 

If $p = \infty$, from part (i) and (ii) we know that for any $m = 1,\ldots, n-1$, with probability $1 - O(n^{-4})$,
\begin{equation*}
\| \Delta E_{22}^{(m)} D^{-1} \|_{\infty, \infty} \le C_0 \sqrt{\log n} \| \ddiag(D^{-1}) \|_1 \le 2C_0 \sqrt{\log n} \| \ddiag(D_\lambda^{-1}) \|_1,
\end{equation*}
which is smaller than $1/12$ under Assumption \ref{assmp::lambda}, where the constant $c_0$ is taken to be small enough. Consequently, the desired bound (\ref{ineqn::sparse}) follows. 
\end{proof}

\subsection{Additional Proofs in Section \ref{sec::match}}

\begin{proof}[Proof of Theorem \ref{thm::match}]
We follow our plan as outlined in Section \ref{sec::match}. We will prove that with probability $1 - O(n^{-1})$, (i) $\tilde{\lambda} > (\lambda_1 + \lambda_2)/2$; (ii) for any unit vector $z$ orthogonal to $\tilde{u}$, $z^* \tilde{A} z \le (\lambda_1 + \lambda_2)/2$. \\

Let us first consider (i). We claim the following identity
\begin{equation}\label{eqn::tildeA1}
\tilde{u}^* \tilde{A} \tilde{u} = (\lambda_1 + E_{11}) + E_{12}q^\infty,
\end{equation}
which helps us to relate $\tilde{\lambda}$ to $\lambda_1$. To verify this identity, we substitute (\ref{eqn::tildeu}) into the left-hand side of (\ref{eqn::tildeA1}):
\begin{align*}
&~( u + U_\bot q^\infty)^* (A+E) ( u + U_\bot q^\infty) \\
= & ~ (\lambda_1 +E_{11}) + (q^\infty)^* \diag(\lambda_2,\ldots, \lambda_n)q^\infty + (q^\infty)^* E_{21} + E_{12}q^\infty + (q^\infty)^* E_{22}q^\infty,
\end{align*}
where we used the notation in (\ref{eqn::EblockMat}). Since $q^\infty$ is a solution to the quadratic equations (\ref{eqn::quad}), we have
\begin{equation*}
(\lambda_1 + E_{11}) \| q^\infty \|_2^2 - (q^\infty)^* (\diag(\lambda_2,\ldots, \lambda_n) + E_{22}) q^\infty = (q^\infty)^* E_{21} - \| q^\infty \|_2^2 E_{12}q^\infty.
\end{equation*}
After rearrangement and substitution, we can simplify:
\begin{equation*}
( u + U_\bot q^\infty)^* (A+E) ( u + U_\bot q^\infty) = (1 + \| q^\infty \|_2^2) \big( (\lambda_1 + E_{11}) + E_{12}q^\infty \big).
\end{equation*}
Dividing both sides by $1 + \| q^\infty \|_2^2$ yields (\ref{eqn::tildeA1}). To prove (i), we need to show $|E_{11} + E_{12}q^\infty| < ( \lambda_1-\lambda_2)/2$. Lemma \ref{lem::indct} says that with probability $1 - O(n^{-2})$, $\| D^{-1} E_{21} \|_{p'} \le 1/2$, so by H\"{o}lder's inequality
\begin{equation*}
| E_{12}q^\infty | = |(D^{-1}E_{21})^* (Dq^\infty) | \le \| D^{-1} E_{21} \|_{p'} \| D q^\infty \|_p \le n^{1/p} \| Dq^\infty \|_\infty / 2.
\end{equation*}
Recall that in Theorem \ref{thm::indctQs}, we established $\| D q^\infty \|_\infty \le 3\kappa_2 \sqrt{\log n}$; and by Lemma \ref{lem::indct} part (i), $ | E_{11} | \le C_0 \sqrt{\log n}$, where both bounds succeed with probability at least $1 - O(n^{-1})$. Therefore, 
\begin{equation}\label{ineqn::match4}
| E_{11} + E_{12}q^\infty | \le C_0 \sqrt{\log n} + 3\kappa_2 \sqrt{\log n}\, n^{1/p} / 2 < \delta/2,
\end{equation}
with sufficiently small constant $c_0$ in Assumption \ref{assmp::lambda}. This implies $\tilde{\lambda} > \lambda_1 - \delta / 2 = (\lambda_1 + \lambda_2)/2$. \\
%

To prove claim (ii), we will show that with probability $1 - O(n^{-1})$, for any $a \in \bb{C}$ with $|a| \le \sqrt{2}/4$ and any unit vector $x \in \bb{S}^{n-2}$, the unit vector 
\begin{equation}\label{eqn::zform}
z = au_1 + \sqrt{ 1 - |a|^2 }\, U_\bot x
\end{equation}
satisfies $z^* \tilde{A} z \le \lambda_2 + \delta/2$. This readily implies our claim (ii), because $\|q^\infty \|_2 \le 1/4$ by Corollary \ref{cor::q2norm}. In fact, for any $z$ of the form (\ref{eqn::zform}) with $z \bot \tilde{u}$,
\begin{equation*}
\langle \tilde{u}, z \rangle = 0 ~ \Leftrightarrow ~ a + \sqrt{1 - |a|^2} \langle q^\infty, x \rangle = 0  ~ \Rightarrow  ~ |a| (1 - |a|^2)^{-1/2} \le \|q^\infty \|_2.
\end{equation*}
So it must satisfy $|a| (1 - |a|^2)^{-1/2} \le 1/4$, and thus $|a| \le \sqrt{2}/4$. Hence, we only need to bound $z^* \tilde{A} z$ for those $z \in \bb{S}^{n-1}$ with $|a| \le \sqrt{2}/4$. 

Before manipulating $z^* \tilde{A} z$, let us make a useful observation: when we replace $z$ with $e^{i \theta}z$ $(\theta \in \bb{R})$, the value of $z^* \tilde{A} z $ does not change, so we can assume without loss of generality that $a \in \bb{R}^+$. Now we expand $z^* \tilde{A} z$,
\begin{align*}
& ~~~~ z^* ( A + E ) z \\
&= a^2 (\lambda_1 + E_{11}) + (1 - a^2)x^* \big( \diag(\lambda_2,\ldots,\lambda_n)+E_{22} \big) x +  2 a \sqrt{1 - a^2}\, \Re(E_{12}x) \\
&= a^2 \lambda_2 +( 1 - a^2) \left( x^* \big( \diag(\lambda_2,\ldots,\lambda_n)+E_{22} \big) x + 2a/\sqrt{1-a^2}\,  \Re(E_{12}x) \right)   \\
&+  a^2 (\delta + E_{11}).
\end{align*}
From Lemma \ref{lem::indct} part (i), with probability $1 - O(n^{-4})$ we have $|E_{11}| \le C_0\sqrt{\log n}$, which is smaller than $\delta$ for suitable choice of constant $c_0$. Thus the last term above has an upper bound $a^2 (\delta + E_{11}) \le \delta / 4$. If we can show that, with probability $1 - O(n^{-1})$, for any $a \in [0,\sqrt{2}/4]$ and unit vector $x \in \bb{S}^{n-1}$,
\begin{equation}\label{ineqn::opta}
x^* \big( \diag(\lambda_2,\ldots,\lambda_n)+E_{22} \big) x + 2a/\sqrt{1-a^2}\,  \Re(E_{12}x) \le  \lambda_2 + \delta/4, 
\end{equation}
then the desired bound $z^* \tilde{A} z \le \lambda_2 + \delta/2$ follows immediately. Let us write $\tau_a = 2a/\sqrt{1-a^2}$ and $D_\mu = \diag(\mu_1,\ldots, \mu_{n-1})$ where $\mu_j = \delta/4 + \lambda_2 - \lambda_{1+j}$. Since $\| x \|_2 = 1$, we can rearrange (\ref{ineqn::opta}) into an equivalent form:
\begin{equation}\label{ineqn::muPtb}
x^* E_{22} x + \tau_a \| x \|_2 \Re(E_{12}x) \le x^* D_\mu x.
\end{equation}
Note that the left-hand side of (\ref{ineqn::opta}) attains maximum if and only if $a$ is one of the two endpoints of $[0,\sqrt{2}/4]$. This says we only need to apply Theorem \ref{thm::eigvalPtb} for $a=0$ and $a = \sqrt{2}/4$. To do so, we check its condition: since $\mu_j = \delta/4 + \lambda_2 - \lambda_{1+j} \ge (\lambda_1 - \lambda_{1+j})/4$, we have
\begin{equation*}
\| \ddiag(D_\lambda^{-1}) \|_{p/(p-2)} \ge \| \ddiag(D_\mu^{-1}) \|_{p/(p-2)}/4.
\end{equation*}
We require the absolute constant $c_0 \le c_0'/4$, where constant $c'_0$ appears in Assumption \ref{assmp::mu}. With this choice the condition for Theorem \ref{thm::eigvalPtb} is satisfied. Applying Theorem \ref{thm::eigvalPtb} twice, for $a = 0$ and $a = \sqrt{2}/4$, we deduce that (\ref{ineqn::muPtb}) holds with probability $1 - O(n^{-1})$ for any $x \in \bb{C}^{n-1}$ and $a \in [0, \sqrt{2}/4 ]$. This completes the proof of claim (ii).

\end{proof}

\section*{Acknowledgement}
The author would like to thank Jianqing Fan, Ramon van Handel, Jiaoyang Huang, Kaizheng Wang, Ziquan Zhuang, and Yantao Wu for helpful discussions.

\bibliographystyle{ims}
\bibliography{references}

\end{document}